\documentclass[letterpaper, 10pt]{article}
\usepackage{amsthm,amsmath,amsfonts,amssymb}
\usepackage{bbm}
\usepackage{graphicx}

\newtheorem{thm}{Theorem}[section]
\newtheorem{our_thm}{Theorem}
\newtheorem{cor}[thm]{Corollary}
\newtheorem{lem}[thm]{Lemma}
\newtheorem{prop}[thm]{Proposition}
\newtheorem{conj}[thm]{Conjecture}
\newtheorem{prob}[thm]{Problem}
\newtheorem{clm}[thm]{Claim}
\theoremstyle{definition}
\newtheorem{defn}[thm]{Definition}
\theoremstyle{remark}
\newtheorem{rem}[thm]{Remark}

\newcommand{\mP}{\mathcal{P}}
\newcommand{\mG}{\mathcal{G}}
\newcommand{\mH}{\mathcal{H}}
\newcommand{\mQ}{\mathcal{Q}}
\newcommand{\mS}{\mathcal{S}}
\newcommand{\mF}{\mathcal{F}}
\newcommand{\mB}{\mathcal{B}}
\newcommand{\bd}{\mathbf{d}}
\newcommand{\od}{\bar{d}}
\newcommand{\hG}{\hat{G}}

\newcommand{\IV}[1]{\mathbf{1}_{#1}}

\newcommand{\Gnd}[2]{\mG_{#1,#2}}
\newcommand{\GnD}[1]{\Gnd{#1}{d}}
\newcommand{\Gtwond}{\GnD{2n}}
\newcommand{\GNd}[1]{\Gnd{n}{#1}}
\newcommand{\GND}{\GNd{d}}

\newcommand{\Gnp}[2]{\mG(#1,#2)}

\newcommand{\GNp}[1]{\Gnp{n}{#1}}
\newcommand{\GNP}{\GNp{p}}

\newcommand{\PerfMatch}{\mathcal{PM}}
\newcommand{\Hamiltonicity}{\mathcal{HAM}}
\newcommand{\EdgeConn}[1]{\mathcal{EC}_{#1}}
\newcommand{\VertConn}[1]{\mathcal{VC}_{#1}}

\renewcommand{\Pr}{\mathbb{P}}
\newcommand{\Prob}[1]{\Pr\left[#1\right]}
\newcommand{\cProb}[2]{\Pr\left[ \left. #1 \;\right\vert #2 \right]}

\newcommand{\sExp}[2]{\mathbb{E}_{#1}\left[#2\right]}
\newcommand{\Exp}[1]{\sExp{}{#1}}
\newcommand{\Bin}{\mathbf{Bin}}

\title{Local resilience and Hamiltonicity Maker-Breaker games in random regular graphs}
\author{
Sonny Ben-Shimon\thanks{School of Computer Science, Raymond and Beverly Sackler Faculty of Exact Sciences, Tel Aviv University, Tel Aviv 69978, Israel. E-mail: sonny@tau.ac.il. Research partially supported by a Farajun Foundation Fellowship.}
\and Michael Krivelevich\thanks{School of Mathematical Sciences, Raymond and Beverly Sackler Faculty of Exact Sciences, Tel Aviv University, Tel Aviv 69978, Israel. E-mail: krivelev@tau.ac.il. Research supported in part by a USA-Israel BSF grant, by a grant from the Israel Science Foundation and by a Pazy Memorial Award.}
\and Benny Sudakov\thanks{Department of Mathematics, UCLA, Los Angles 90005, CA, USA. Email: bsudakov@math.ucla.edu. Research supported in part by NSF CAREER award DMS-0812005 and by USA-Israeli BSF grant.}}
\begin{document}
\maketitle
\begin{abstract}
For an increasing monotone graph property $\mP$ the \emph{local resilience} of a graph $G$ with respect to $\mP$ is the minimal $r$ for which there exists of a subgraph $H\subseteq G$ with all degrees at most $r$ such that the removal of the edges of $H$ from $G$ creates a graph that does not possesses $\mP$. This notion, which was implicitly studied for some ad-hoc properties, was recently treated in a more systematic way in a paper by Sudakov and Vu. Most research conducted with respect to this distance notion focused on the Binomial random graph model $\GNP$ and some families of pseudo-random graphs with respect to several graph properties such as containing a perfect matching and being Hamiltonian, to name a few. In this paper we continue to explore the local resilience notion, but turn our attention to random and pseudo-random \emph{regular} graphs of constant degree. We investigate the local resilience of the typical random $d$-regular graph with respect to edge and vertex connectivity, containing a perfect matching, and being Hamiltonian. In particular, we prove that for every positive $\varepsilon$ and large enough values of $d$ with high probability the local resilience of the random $d$-regular graph, $\GND$, with respect to being Hamiltonian is at least $(1-\varepsilon)d/6$. We also prove that for the Binomial random graph model $\GNP$, for every positive $\varepsilon>0$ and large enough values of $K$, if $p>\frac{K\ln n}{n}$ then with high probability the local resilience of $\GNP$ with respect to being Hamiltonian is at least $(1-\varepsilon)np/6$. Finally, we apply similar techniques to Positional Games and prove that if $d$ is large enough then with high probability a typical random $d$-regular graph $G$ is such that in the unbiased Maker-Breaker game played on the edges of $G$, Maker has a winning strategy to create a Hamilton cycle.
\end{abstract}

\section{Introduction}\label{s:Introduction}
Let $\mP$ be some graph property. A basic question in extremal combinatorics asks to compute or to estimate the ``distance'' of a graph $G$ from possessing (or not possessing) the property $\mP$. The following definition comes to mind.
\begin{defn}
The \emph{global resilience} of a graph $G=(V,E)$ w.r.t. $\mP$ is
$$r_g(G,\mP)=\min\{r:\exists F\subseteq \binom{V}{2}\;\mbox{such that}\;|F|=r\;\mbox{and}\; G'=(V,E\triangle F)\notin \mP\}.$$
\end{defn}
Hence, the \emph{global resilience} of $G$ with respect to $\mP$ is the minimal number of additions and removals of edges from $G$ such that the resulting graph does not possess $\mathcal{P}$. This is in fact the \emph{edit distance} of $G$ with respect to $\overline{\mP}$, and using this terminology, one can state the celebrated theorem of Tur\'{a}n \cite{Tur41} (in the case that $r$ divides $n$) as ``The complete graph $K_n$ on $n$ vertices has global resilience $\frac{n}{2}(\frac{n}{r}-1)$ with respect to being $K_{r+1}$-free''.

For some graph properties such as being connected or being Hamiltonian, the removal of all edges incident to a vertex of minimum degree is enough to destroy them, hence supplying a trivial upper bound on the global resilience. For such properties the notion of global resilience does not seem to convey what one would expect from such a distance measure. In a recent paper, Sudakov and Vu \cite{SudVu2008} initiated the systematic study of the following related notion which conditions not on the maximal number of edges in the graph $H$, but on its maximum degree $\Delta(H)$.
\begin{defn}\label{d:locres}
The \emph{local resilience} of a graph $G=(V,E)$ w.r.t. $\mP$ is
$$r_\ell(G,\mP)=\min\{r:\exists H\subseteq \binom{V}{2}\;\mbox{such that}\;\Delta(H)=r\;\mbox{and}\; G\triangle H\notin \mP\}.$$
\end{defn}
So, for local resilience, an additional constraint of a bounded number of editions done on edges incident to a single vertex is required. Using this definition, for example, the classical theorem of Dirac (see, e.g., \cite[Theorem 10.1.1]{Die2005}) can be rephrased as ``$K_n$ has local resilience $\left\lfloor n/2\right\rfloor$ with respect to being Hamiltonian''. As has already been pointed out in \cite{SudVu2008} there seems to be a duality between the global or local nature of the graph property at hand and the type of resilience that is more natural to consider. More specifically, global resilience seems to be a more appropriate notion for studying \emph{local} properties (e.g. containing a $K_k$ subgraph), whereas for global properties (e.g. being Hamiltonian), the study of local resilience appears to be more natural.

In this paper we will focus of the local resilience of monotone increasing graph properties $\mP$ (i.e. properties that are kept under the addition of edges). This implies that one should only consider the removal of edges, as addition of edges cannot eliminate $\mP$. Definition \ref{d:locres} hence takes the following simpler form
$$r_\ell(G,\mP)=\min\{r:\exists H\subseteq G\;\mbox{such that}\;\Delta(H)=r\;\mbox{and}\; G\setminus H\notin \mP\}.$$

Our results will mostly deal with local resilience of \emph{random graphs}. The most widely used random graph model is the Binomial random graph, $\GNP$. In this model we start with $n$ vertices, labeled, say, by $\{1,\ldots,n\}=[n]$, and select a graph on these $n$ vertices by going over all $\binom{n}{2}$ pairs of vertices, deciding independently with probability $p$ for a pair to be an edge. The model $\GNP$ is thus a probability space of all labeled graphs on the vertex set $[n]$ where the probability of such a graph, $G=([n],E)$, to be selected is $p^{|E|}(1-p)^{\binom{n}{2}-|E|}$. This product probability space gives us a wide variety of probabilistic tools to analyze the behavior of various random graph properties of this probability space. (See monographs \cite{Bol2001} and \cite{JanLucRuc2000} for a thorough introduction to the subject of random graphs). In this paper we will mostly consider a different random graph model. Our probability space, which is denoted by $\GND$ (where $dn$ is even), is the uniform space of all $d$-regular graphs on $n$ vertices labeled by the set $[n]$. In this model, one cannot apply the techniques used to study $\GNP$ as these two models do not share the same probabilistic properties. Whereas the appearances of edges in $\GNP$ are independent, the appearances of edges in $\GND$ are not. Nevertheless, many results obtained thus far for the random regular graph model $\GND$ are in some sense ``equivalent'' to the results obtained in $\GNP$ with suitable expected degrees, namely, $d=np$. This relation between the two random graph models was partially formalized by Kim and Vu in \cite{KimVu2004}. The main results of this current paper are yet further examples of this connection between the two random graph models as will become apparent in the subsequent sections and in the appendix. The interested reader is referred to \cite{Wor99} for a thorough survey of the random regular graph model $\GND$. We will sometimes abuse notation a bit by using the same notation for both the distribution over graphs on a set of vertices and a random graph sampled from this distribution, but which of the two is meant should be clear from the context.

\subsection{Previous work}\label{ss:PrevWork}
A similar notion to that of local resilience of graph properties was first mentioned in a paper of Kim and Vu \cite{KimVu2004} where their main incentive was to prove some formal relation that ties the classic Binomial random graph model $\GNP$ and $\GND$ when $d=np$ and $p$ is large enough. Local resilience as defined in Definition \ref{d:locres} was implicitly used in a paper of Kim, Sudakov and Vu \cite{KimSudVu2002} where it was proved that both typical random graphs $\GNP$ and $\GND$, for appropriate values of $p$ and $d$, do not have non-trivial automorphisms, thus settling a conjecture of Wormald.

Recently, the third author and Vu in \cite{SudVu2008} initiated a systematic study of the local resilience of graph properties, for it was apparent that this notion plays a central role in several classical results in extremal graph theory and deserved some attention on its own right. In a relatively short period of time quite a few research papers on the subject followed (see e.g. \cite{DelEtAl2008, FriKri2008, KriLeeSud2010}). All of the above mentioned papers dealt more specifically with the local resilience of random graphs with respect to several graph properties. We note that some of the following results were actually proved with respect to pseudo-random graphs, much like will be done in the current paper, but in order to omit some technical definitions, we will solely state the results for random graphs. In this context one is looking on the typical behavior of $r_\ell(\mathcal{G})$ where $\mathcal{G}$ is some random-graph model. We will specifically cite some of the above mentioned results which are closely related to the type of questions that this paper deals with, and, actually, were the main motivation for it.

Let $\PerfMatch$ denote the graph property of containing a perfect matching. In their paper, Sudakov and Vu \cite{SudVu2008} proved that there exists an absolute constant $C>0$ such that if $p\geq \frac{C\log n}{n}$ and $n$ is even then with high probability (or w.h.p. for brevity)
\footnote{In this context we mean that the mentioned event holds with probability tending to $1$ as $n$, the number of vertices, goes to infinity.}
$r_\ell(\GNP,\PerfMatch)=\frac{np}{2}(1\pm o(1))$, hence pretty much settling the local resilience question of random graphs with respect to this property. Another natural property to consider is the graph property of being Hamiltonian which we denote by $\Hamiltonicity$. Note that for every graph $G$ we have the following trivial lower bound
\begin{equation}
r_\ell(G,\Hamiltonicity)\leq r_\ell(G,\PerfMatch).
\end{equation}
\begin{rem}
When $n$ is odd one can define the property $\PerfMatch$ as containing a matching that misses one vertex, and an analogous result to the one just mentioned can be similarly derived.
\end{rem}
Still in \cite{SudVu2008}, Sudakov and Vu showed that there exists an absolute constant $C$ such that for every $\delta,\varepsilon>0$, if $p\geq \frac{C\log^{2+\delta} n}{n}$ then w.h.p. $r_\ell(G,\Hamiltonicity) \geq (1-\varepsilon)np/2$. Frieze and Krivelevich in \cite{FriKri2008} studied this problem for the range of $p$ ``right after'' $\GNP$ becomes Hamiltonian w.h.p., but the lower bound they obtained in this range is weaker. They proved that there exist absolute constants $\alpha,C>0$ such that for every $p\geq \frac{C\log n}{n}$ w.h.p. $r_\ell(\GNP,\Hamiltonicity)\geq\alpha np$. It is plausible that w.h.p. $r_\ell(\GNP,\Hamiltonicity)=(1/2\pm o(1))np$ as soon as $p\gg\frac{\log n}{n}$, but the above mentioned results still leave a gap to fill. In this work we make some progress on this front, but, alas, we are unable to close the gap completely.

\subsection{Our results}\label{ss:OurResults}
As previously mentioned, in this work we continue to explore the notion of local resilience of random and pseudo-random graphs, but our focus is shifted to the regular case. Our first result deals with the connectivity property. Recall that for the $\GNP$ model we need $p\geq\frac{\log n}{n}$ for the graph to be w.h.p. connected, whereas in the regular case, $\GND$, taking $d\geq 3$ suffices for the random $d$-regular graph to be w.h.p. connected. We start by showing that in this same range not only is $\GND$ w.h.p. connected, it is also somewhat resilient with respect to this property. Let $\EdgeConn{k}$ and $\VertConn{k}$ denote the graph properties of being $k$-edge connected and $k$-vertex connected, respectively. We start with the edge connectivity case.

\begin{our_thm}\label{t:edge_conn}
For every fixed $d\geq 3$ the following holds w.h.p.:
\begin{enumerate}
\item\label{i:edge_conn1} If $\frac{d}{2}+\sqrt{d}\leq k\leq d$ then $r_\ell(\GND,\EdgeConn{k})=d-k+1$;
\item\label{i:edge_conn2} If $1\leq k<\frac{d}{2}+\sqrt{d}$ then $\frac{d}{2}-\sqrt{d}< r_\ell(\GND,\EdgeConn{k})\leq\min\{d-k+1,\frac{d}{2}+4\sqrt{d\ln d}\}$.
\end{enumerate}
\end{our_thm}
Theorem \ref{t:edge_conn} demonstrates an interesting threshold phenomenon that happens as we decrease $k$ from $d$ to $1$ at around $k=\frac{d}{2}$. It is apparent that if we want our graph to stop being $k$-edge connected for $k\geq d/2+\sqrt{d}$, then the best we can do is to remove the edges incident to a single vertex, but when $k$ goes below $d/2-4\sqrt{d\ln d}$ that is no longer the case, as one can find cuts where each vertex does not participate in ``too many'' edges of the cut. The same phenomenon happens in the case of vertex connectivity.
\begin{our_thm}\label{t:vert_conn}
There exists an integer $d_0>0$ such that for every fixed $d\geq d_0$ the following holds w.h.p.:
\begin{enumerate}
\item\label{i:vert_conn1} If $\frac{d}{2}+\sqrt{d}\leq k\leq d$ then $r_\ell(\GND,\VertConn{k})=d-k+1$;
\item\label{i:vert_conn2} If $1\leq k<\frac{d}{2}+\sqrt{d}$ then $\frac{d}{2}-\sqrt{d}<r_\ell(\GND,\VertConn{k})\leq\min\{d-k+1,\frac{d}{2}+4\sqrt{d\ln d}\}$.
\end{enumerate}
\end{our_thm}

The next property that we investigate is that of containing a perfect matching. This result, in fact, is inspired by the corresponding results for $\GNP$ of Sudakov and Vu in \cite{SudVu2008}, although some modifications of their proof were needed for it to apply to the case of fixed values of $d$. We show that for large enough fixed $d$, the local resilience of the typical random regular graph with respect to the containment of a perfect matching is concentrated around the value $d/2$.
\begin{our_thm}\label{t:perf_match}
There exists an integer $d_0>0$ such that for every fixed $d\geq d_0$ w.h.p.
$$\frac{d}{2}-10\sqrt{d\ln d}-4\sqrt{d}< r_\ell(\Gtwond,\PerfMatch)\leq \frac{d}{2}+2\sqrt{d\ln d}+2.$$
\end{our_thm}
Actually, Theorems \ref{t:edge_conn}, \ref{t:vert_conn} and \ref{t:perf_match} are proved in a stronger setting, that of $(n,d,\lambda)$-graphs (which we define in Subsection \ref{ss:ndlambda}). The results above, cited in the context of random regular graphs, are just simple implications when plugging in the known bounds on the second eigenvalue of the typical random regular graph.

We move to the main result of this paper, namely, a lower bound on the local resilience of the typical random regular graph of constant degree with respect to being Hamiltonian.
\begin{our_thm}\label{t:HamResGnd}
For every $\varepsilon>0$ there exists an integer $d_0(\varepsilon)>0$ such that for every fixed integer $d \geq d_0$ w.h.p.
$$r_\ell(\GND,\Hamiltonicity)\geq (1-\varepsilon)d/6.$$
\end{our_thm}

Following the footsteps traced by the proof of Theorem \ref{t:HamResGnd}, we are able to prove the ``corresponding'' result in the case of the random graph model $\GNP$, hence improving the non explicit constant value $\alpha$ of \cite{FriKri2008} to $\frac{1}{6}-\varepsilon$ for any $\varepsilon>0$.
\begin{our_thm}\label{t:HamResGnp}
For every $\varepsilon>0$ there exists a constant $K(\varepsilon)>0$ such that if $p\geq K\log n/n$ then w.h.p.
$$r_\ell(\GNP,\Hamiltonicity)\geq (1-\varepsilon)np/6.$$
\end{our_thm}

\subsection{Positional games played on $\GND$}
Let $G=(V,E)$ be a graph and consider the following game played on the set of edges $E(G)$. The game is played by two players, \emph{Maker} and \emph{Breaker}, who alternately take turns occupying a previously unclaimed edge of $E(G)$. We assume that Breaker moves first and that the game ends when all edges of the graph have been claimed by either of the players. The game is a win for Maker if and only if the graph spanned by the edges selected by Maker possesses some predefined graph property $\mP$. The graph $G$ is called a \emph{Maker's win} if no matter how Breaker plays, Maker has a strategy (that can be adaptive to Breaker's choices) such that the game ends as a win for Maker, and we denote the family of winning boards for Maker by $\mathcal{M}_{\mP}$. Although the above game is described in game-theoretic terms, it should be noted that these games are perfect-information games and $\mathcal{M}_{\mP}$ is some graph property. It clearly satisfies $\mathcal{M}_{\mP}\subseteq \mP$ and if $\mP$ is monotone increasing (decreasing) then $\mathcal{M}_{\mP}$ is also monotone increasing (decreasing). Furthermore, in the case of monotone increasing graph properties, the game can terminated once the graph spanned by Maker's edges possesses the property regardless of whether all edges have been claimed. This game is a particular case of a general family of combinatorial games called \emph{Positional Games}. Positional Games have attracted more and more attention in the past decade and a thorough introduction with a plethora of results can be found in a recent monograph of Beck \cite{Beck2008}.

One of the seminal results in this field is due to Chv\'{a}tal and Erd\H{o}s \cite{ChvErd78} who proved that $K_n\in \mathcal{M}_\Hamiltonicity$ for large enough values of $n$ (Hefetz and Stich \cite{HefSti2009} proved $n\geq29$ suffices). The monotonicity of $\Hamiltonicity$ leads to the natural question of how sparse can a graph $G\in\mathcal{M}_\Hamiltonicity$ be. Hefetz et. al. \cite{HefEtAl2009, HefEtAlPre} addressed this problem twice. First, they proved that there exists a positive constant $\ell$ such that w.h.p. $\GNP\in\mathcal{M}_\Hamiltonicity$ for every $p\geq\frac{\ln n + (\ln\ln n)^\ell}{n}$. Note that this result is very close to being optimal for if $p=\frac{\ln n + 3\ln\ln n - \omega(1)}{n}$ then w.h.p. $\delta(\GNP)<4$ which directly implies that $\GNP\notin\mathcal{M}_\Hamiltonicity$ (Breaker, having the advantage of the first move, will start by taking more than half of the edges incident to a vertex of degree at most $3$ which exists with high probability, leaving Maker's graph with a vertex of degree at most $1$ and thus without a Hamilton cycle). Second, they showed that for large enough values of $n$ there exists a graph $G\in \mathcal{M}_\Hamiltonicity$ on $n$ vertices with $e(G)\leq 21n$.

We study the Hamiltonicity game played on the edges of random regular graphs. In turns out that using very similar ideas to the ones used in the proof of Theorem \ref{t:HamResGnd} can help demonstrate that a typical graph sampled from $\GND$ for large enough constant values of $d$ is Maker's win for this game.
\begin{our_thm}\label{t:HamGame}
There exists an integer $d_0>0$ such that for every fixed integer $d\geq d_0$ w.h.p.
$$\GND\in\mathcal{M}_\Hamiltonicity.$$
\end{our_thm}
It should be noted that in contrast to the result of \cite{HefEtAlPre} Theorem \ref{t:HamGame} gives a very large and natural family of graphs on which Maker wins the Hamiltonicity game. Moreover, a typical graph in this family will be locally sparse, whereas the construction of the graph of \cite{HefEtAlPre} contains many cliques of large constant size.

Finally, we would like to pinpoint this ideological connection between Maker-Breaker games and the notion of local resilience, and specifically the similar nature of Theorems \ref{t:HamResGnd} and \ref{t:HamGame}. Both theorems imply that even for constant (albeit large) values of $d$ a typical graph sampled from $\GND$ is not only Hamiltonian (see \cite{Bol83,FenFri84,Fri88,RobWor92,RobWor94}), but even an edge-deleting adversary (of somewhat limited power) cannot make the graph non-Hamiltonian. This connection ``in spirit'' of these two notions also leads to similar ideas and techniques which we show can be applied in the proofs in both settings.

\subsection{Organization}\label{ss:Organization}
The rest of the paper is organized as follows. We start with Section \ref{s:Preliminaries} where we state all the needed preliminaries that are used throughout the proofs of our results. Section \ref{s:ConnAndPerfMatch} is devoted to the proofs of Theorems \ref{t:edge_conn}, \ref{t:vert_conn} and \ref{t:perf_match} which share common ideas. In Section \ref{s:Hamiltonicity} we give in detail the full proof of Theorem \ref{t:HamResGnd} which is somewhat more involved than the previous proofs and requires a delicate investigation of the random graph model. As the main focus of this current paper is on random regular graphs of constant degree, we relegate the full proof of Theorem \ref{t:HamResGnp} to the Appendix, where the proof itself follows quite closely that of Theorem \ref{t:HamResGnd}. Section \ref{s:HamGame} is devoted to the proof of Theorem \ref{t:HamGame}, and we conclude the paper with some final remarks and open questions in Section \ref{s:Conclusion}.

\section{Preliminaries}\label{s:Preliminaries}
In this section we provide the necessary background information needed in the course of the proofs of the main results of this paper. We choose the algebraic approach to pseudo-randomness (although this can be readily replaced by many other qualitatively equivalent definitions of pseudo-random graphs) as the transition from $(n,d,\lambda)$-graphs with a large enough spectral gap to random regular graphs is quite standard. We then move to describe in some details some previous results regarding the random graph model $\GND$ when $d$ is fixed. We will also need to introduce the more general setting of random graphs of a specified degree sequence, and although our main results are not stated in this general setting, our proofs of the Hamiltonicity result will rely heavily on this setting. We start with (the mostly standard) notation that will be used throughout this paper.
\subsection{Notation}\label{ss:Notation}
Although this paper mainly deals with graphs where neither loops nor parallel edges are allowed, it will be more convenient to define some of the notation in a more general setting where parallel edges and loops may exist. In order to remove any ambiguity, we refer to an object in the more general setting as a \emph{multigraph}, whereas the term \emph{graph} is strictly reserved for the case where no loops nor parallel edges appear.

Given two multigraphs $M_1=(V,E_1)$ and $M_2=(V,E_2)$ on the same vertex set, we denote by $M_1+M_2=(V,E)$ the multigraph over the same vertex set, $V$, with an edge multiset $E$ taken as the union as multisets of $E_1$ and $E_2$.

Given a graph $G=(V,E)$, the \emph{neighborhood} $N_G(U)$ of a subset $U\subseteq V$ of vertices is the set of vertices defined by $N_G(U)=\{v\notin U\;:\;\{v,u\}\in E\}$, and the degree of a vertex $v$ is $d_G(v)=|N_G(\{v\})|$. We denote by $E_G(U)$ the set of edges of $G$ that have both endpoints in $U$, and by $e_G(U)$ its cardinality. Similarly, for two disjoint subsets of vertices $U$ and $W$, $E_G(U,W)$ denotes the set of edges with an endpoint in $U$ and the other in $W$, and $e_G(U,W)$ its cardinality. We will sometime refer to $e_G(\{u\},W)$ by $d_G(u,W)$. We use the usual notation of $\Delta(G)$ and $\delta(G)$ to denote the respective maximum and minimum degrees in $G$. We say that $H$ is a \emph{subgraph} of $G$, and write $H\subseteq G$ if the graph $H=(V,F)$ has the same vertex set as $G$ and its edge set satisfies $F\subseteq E$. If $V=V_1\cup V_2$ is a partition of the vertex set, we let $G_{V_1,V_2}=(V_1\cup V_2,E_G(V_1,V_2))$ be the induced bipartite subgraph of $G$ with parts $V_1$ and $V_2$. We denote the maximal density of sets of vertices of cardinality at most $k$ by $\rho(G,k)=\max\left\{\frac{e_G(U)}{|U|}\;:\;U\subseteq V \hbox{ s.t. } |U|\leq k\right\}$. Lastly, we will denote by $\ell(G)$ the length of a longest path in $G$.

The main research interest of this paper is the asymptotic behavior of some properties of graphs, when the graph is sampled from some probability measure $\mG$ over a set of graphs on the same vertex set $[n]$, and the number of vertices, $n$, grows to infinity. Therefore, from now on and throughout the rest of this work, when needed we will always assume $n$ to be large enough. We use the usual asymptotic notation. For two functions of $n$, $f(n)$ and $g(n)$, we denote $f=O(g)$ if there exists a constant $C>0$ such that $f(n)\leq C\cdot g(n)$ for large enough values of $n$; $f=o(g)$ or $f\ll g$ if $f/g\rightarrow 0$ as $n$ goes to infinity; $f=\Omega(g)$ if $g=O(f)$; $f=\Theta(g)$ if both $f=O(g)$ and $g=O(f)$.

Throughout the paper we will need to employ bounds on large deviations of random variables. We will mostly use the following well-known bound on the lower and the upper tails of the Binomial distribution due to Chernoff (see e.g. \cite[Appendix A]{AloSpe2008}).
\begin{thm}[Chernoff bounds]\label{t:Chernoff}
Let $X\sim\Bin(n,p)$, then for every $\Delta>0$
\begin{enumerate}
\item\label{i:Chernoff1} $\Prob{X > (1+\Delta)np}<\exp(-np((1+\Delta)\ln(1+\Delta)-\Delta))$;
\item\label{i:Chernoff2} $\Prob{X < (1-\Delta)np}<\exp(-\frac{\Delta^2np}{2})$; 
\item\label{i:Chernoff3} $\Prob{|X-np|>\Delta np}<2\exp(-np((1+\Delta)\ln(1+\Delta)-\Delta))$.
\end{enumerate}
\end{thm}
Lastly, we stress that throughout this paper we may omit floor and ceiling values when these are not crucial to avoid cumbersome exposition.

\subsection{$(n,d,\lambda)$-graphs}\label{ss:ndlambda}
The \emph{adjacency matrix} of a $d$-regular graph $G$ on $n$ vertices labeled by $\{1,\ldots,n\}$, is the $n\times n$ binary matrix, $A=A(G)$, where $A_{ij}=1$ iff $(i,j)\in E(G)$. As $A$ is real and symmetric it has an orthogonal basis of real eigenvectors and all its eigenvalues are real. We denote the eigenvalues of $A$ in descending order by $\lambda_1\geq\lambda_2 \ldots\geq \lambda_n$, where $\lambda_1=d$ and its corresponding eigenvector is $\mathbbm{1}_n$ (the $n\times 1$ all ones vector). Finally, let $\lambda=\lambda(G)=\max\{|\lambda_2(G)|,|\lambda_n(G)|\}$, and call such a graph $G$ an $(n,d,\lambda)$-graph. For an extensive survey of fascinating properties of $(n,d,\lambda)$-graphs the reader is referred to \cite{KriSud2006}. The celebrated expander mixing lemma (see e.g. \cite{AloSpe2008} or \cite{Chung2004}) states roughly that the smaller $\lambda$ is, the more random-like is the graph.
\begin{lem}[The Expander Mixing Lemma - Corollary 9.2.5 in \cite{AloSpe2008}]\label{l:expmixlem}
Let $G=(V,E)$ be an $(n,d,\lambda)$-graph. Then every pair of disjoint subsets of vertices $U,W\subseteq V$ satisfies
\begin{equation*}
\left|e_G(U,W)-\frac{|U||W|d}{n}\right|\leq\frac{\lambda}{n}\sqrt{|U|(n-|U|)|W|(n-|W|)}.
\end{equation*}
\end{lem}

We state two corollaries of the above (see e.g. \cite[Section 4.2]{Chung2004}) which will be applied in the succeeding sections.
\begin{cor}\label{c:expmixlem2}
Let $G=(V,E)$ be an $(n,d,\lambda)$-graph. Then every subset of
vertices $U\subseteq V$ satisfies
$$e(U,V\setminus U)\geq\frac{(d-\lambda)|U|(n-|U|)}{n}.$$
\end{cor}

\begin{cor}\label{c:expmixlem3}
Let $G=(V,E)$ be an $(n,d,\lambda)$-graph. Then every subset of vertices $U\subseteq V$ satisfies
$$e(U)\leq\frac{d}{n}\binom{|U|}{2}+\frac{\lambda}{n}|U|\left(n-\frac{|U|}{2}\right).$$
\end{cor}

\subsection{Random regular graphs}
When thinking about a random regular graph model, a natural choice of probability space, which we denote by $\GND$, is to fix a base set of $n$ vertices and to sample uniformly a $d$-regular graph over this vertex set. This random graph model attracted much attention, and several techniques were developed in order to explore its properties. In this section, we will simply state some of the known results for this model without discussing them nor their proofs (which can be found in \cite{Wor99}). We start with the following easy observation. By symmetry, for every pair of vertices $u,v\in V$
\begin{equation}\label{e:reg_edge_prob}
\Prob{\{u,v\}\in E(\GND)}=\frac{d}{n-1}.
\end{equation}

In light of Subsection \ref{ss:ndlambda}, a possible way to go about proving results on the random graph model $\GND$ is to compute or to estimate the typical value of $\lambda(\GND)$ and then to use the properties of $(n,d,\lambda)$-graphs. In our context, in light of Lemma \ref{l:expmixlem}, we would like to have $d$-regular graphs with \emph{spectral gap}, $d-\lambda(G)$, as large as possible. Friedman, confirming a conjecture of Alon, showed that a typical random $d$-regular graph has a spectral gap which closely matches the upper bound provided by the Alon-Boppana bound (see e.g. \cite{Nil91}), hence providing an accurate evaluation of the second eigenvalue $\lambda(G)$.
\begin{thm}[Friedman \cite{Fri2008}]\label{t:Fri2008}
For every $\varepsilon>0$ and fixed $d\geq 3$ w.h.p.
\begin{equation}\label{e:rhobound}
\lambda(\GND)\leq 2\sqrt{d-1}+\varepsilon.
\end{equation}
\end{thm}

With Lemma \ref{l:expmixlem} and Theorem \ref{t:Fri2008} at hand, the following theorem is an immediate consequence.
\begin{thm}\label{t:twosets_sqrtn}
For every fixed integer $d\geq 3$, if $G=(V,E)$ is sampled from $\GND$ then w.h.p. every pair of subsets of vertices $A,B\subseteq V$ satisfies
\begin{equation}\label{e:oneset}
\left|e(A,B)-\frac{|A||B|d}{n}\right|\leq 2\sqrt{d|A||B|}.
\end{equation}
\end{thm}

The probability space $\GND$ may be a natural probability space to consider, but unfortunately, the inherent dependence of the appearance of edges in a graph sampled from this space creates many technical difficulties. It is sometimes more convenient to work with a different probability space that is in some sense equivalent (for our purposes) to $\GND$, where this equivalence is defined as follows.
\begin{defn}
Let $\mathcal{A}=(A_n)_{n=1}^\infty$ and $\mB=(B_n)_{n=1}^\infty$ be two sequence of probability measures, such that for every natural $n$, $A_n$ and $B_n$ are defined on the same measurable space $(\Omega_n,\mathcal{F}_n)$. We say that $\mathcal{A}$ and $\mB$ are \emph{contiguous} if for every sequence of sets $X_n\in\mathcal{F}_n$,
$$\lim_{n\rightarrow\infty}A_n(X_n)=0\Longleftrightarrow \lim_{n\rightarrow\infty}B_n(X_n)=0.$$
\end{defn}
Our probability measure will be the one induced by some random graph distribution over a fixed set of vertices. Following the notation of \cite[Chapter 9.5]{JanLucRuc2000} we denote contiguity of two random graphs $\mG_n$ and $\mQ_n$ (on the same vertex set) by $\mG_n\approx \mH_n$. By $\mG_n + \mQ_n$ we mean the random multigraph obtained by the union of the two graphs, and by $\mG_n\oplus \mQ_n$ the random graph obtained by taking the union  conditioned on the resulting graph being simple. In particular, we will make use the following results (which were later generalized in a uniform way in \cite{GreEtAl2002}) on the contiguity of the random regular graph probability measure $\GND$ and of the sums of random regular graphs of appropriate degrees.
\begin{thm}[Janson \cite{Jan95}]\label{t:Jan95} For every two fixed integers
$d_1, d_2\geq 3$, $$\GNd{d_1}\oplus\GNd{d_2}\approx\GNd{d_1+d_2}.$$
\end{thm}
\begin{thm}[Kim and Wormald \cite{KimWor2001}]\label{t:KimWor2001} Let $\mH_n$ denote the uniform probability space of all Hamiltonian cycles on a set of $n$ fixed vertices then $$\mH_n\oplus\mH_n\approx\GNd{4}.$$
\end{thm}
We move on to the more general setting of random graphs with a given degree sequence. Let $\GNd{\bd}$ be the uniform probability space over all graphs on vertex set $V$ of size $n$ with degree sequence $\bd=\{d_v\}_{v\in V}$. We call such a sequence $\bd$ \emph{graphic} if there exists at least one graph with this degree sequence. Note that not all degree sequences are graphic. For one, the sum of degrees must always be even. Although our main focus in this paper is the random regular graph model, we will resort to the study of this more general setting towards proving some of our results below. Denote by $\od=\frac{1}{n}\sum_{v\in V}d_v$ the average degree, and by $D$ the maximum degree in this degree sequence. The following result, due to McKay \cite{McK85}, estimates the probability that a random graph with a given degree sequence is edge-disjoint from some given bounded degree graph on the same vertex set.
\begin{thm}[McKay \cite{McK85}]\label{t:McK85}
For every graphic degree sequence $\bd$ with $1\leq D\ll \sqrt n$, if $G_0$ is a graph on $n$ vertices of maximum degree $\Delta(G_0)=O(1)$, then
$$(1-o(1))\exp\left(-\gamma - \gamma^2-\nu +o(1)\right) \leq \Prob{E(\GNd{\bd})\cap E(G_0)=\emptyset} \leq (1+o(1))\exp\left(-\gamma - \gamma^2-\nu +o(1)\right),$$
where $\gamma=\frac{1}{\od n}\sum_{v\in V}\binom{d_v}{2}$ and $\nu=\frac{1}{\od n}\sum_{uv\in E(G_0)}d_u d_v$.
\end{thm}
As a direct consequence we get the following corollary that states that events that occur with negligible probability in $\GNd{d_1}+\GNd{d_2}$ will occur with negligible probability in $\GNd{d_1}\oplus\GNd{d_2}$.
\begin{cor}\label{c:rand_reg_graphs_union}
For every two integers $d_1, d_2\geq 3$, if $\mP$ is a graph property such that $\Prob{\GNd{d_1}+\GNd{d_2}\in \mP}=o(1)$, then $$\Prob{\GNd{d_1}\oplus\GNd{d_2}\in\mP}=o(1).$$
\end{cor}
The proof of Corollary \ref{c:rand_reg_graphs_union} is immediate from Theorem \ref{t:McK85}, as it guarantees that conditioning on the event that the graphs sampled from $\GNd{d_1}$ and $\GNd{d_2}$ are edge disjoint can increase the probability of such an event by a constant (that depends on $d_1$ and $d_2$) multiplicative factor. In turn, applying Theorem \ref{t:Jan95} enables us to study the properties of random regular graphs (of fixed degree) by generating the graph in two phases, where in each phase we generate a random regular graph (of smaller degree), and we can also ``disregard'' multiple edges, as we will be interested only in events which appear with probability tending to $0$ as $n$ grows (or their complement).

The following is a well-known asymptotic property of $\GNd{\bd}$ (see e.g. \cite{Wor99}) which states that w.h.p. any constant size subset of vertices contains at most one cycle. Recall that for any graph $G=(V,E)$ we denote by
\begin{equation*}
\rho(G,\tau)=\max\left\{\frac{e_G(U)}{|U|}\;:\;U\subseteq V \hbox{ s.t. } |U|\leq \tau\right\}.
\end{equation*}

\begin{thm}\label{t:GND_const_density}
Let $\bd=\{d_v\}_{v\in V}$ be a graphic degree sequence such that $D=O(1)$ and let $\tau=O(1)$, then w.h.p. $\rho(\GNd{\bd},\tau)\leq 1$.
\end{thm}

We would like to compute $\Prob{\{u,v\}\in E(\GNd{\bd})}$ for two fixed vertices $u,v\in V$ similarly to \eqref{e:reg_edge_prob}. We start with the following definition.
\begin{defn}
Let $G=(V,E)$ and $G'=(V,E')$ be two graphs on the same vertex set. We write
\begin{equation}\label{e:defswithcdif}
G\sim G'\;\Leftrightarrow\; \exists \{v_1,v_2\},\{u_1,u_2\}\in E \quad E'=E\setminus\{\{v_1,v_2\},\{u_1,u_2\}\}\cup\{\{v_1,u_1\},\{v_2,u_2\}\},
\end{equation}
that is, $G\sim G'$ if $G$ and $G'$ differ only by a single simple \emph{switch} of edges.
\end{defn}
Note that a simple switch operation does not affect the degree sequence of the vertices.
\begin{prop}\label{p:edgeprobGNDbar}
Let $\bd=\{d_v\}_v\in V$ be a graphic degree sequence such that $D=O(1)$, then for every distinct $u,v\in V$,
\begin{equation}\label{eq:edgeprobGNDbar}
(1-o(1))\frac{d_u d_v-d_u-d_v}{\od n + d_ud_v - 2d_u - 2d_v} \leq \Prob{\{u,v\}\in E(\GNd{\bd})}\leq \frac{d_u d_v}{\od n+d_ud_v-(D+1)(d_u+d_v)}.
\end{equation}
\end{prop}
\begin{proof}
Fix a pair of vertices $u$ and $v$, and let
\begin{eqnarray*}
A&=&\{G\in\GNd{\bd}\;:\;\{u,v\}\in E(G)\};\\
B&=&\{G\in\GNd{\bd}\;:\;\{u,v\}\notin E(G)\}.
\end{eqnarray*}

Denote by $\mF$ be the auxiliary bipartite graph with vertex set $A\cup B$ where two vertices of this graph are connected by an edge if the corresponding graphs differ by a simple switch. The graph $\mF$ is undirected, as a simple switch is clearly reversible. For every $G\in\GNd{\bd}$ we denote by $r(G)$ its degree in $\mF$, and thus
\begin{equation}\label{e:switchgraphdeg}
\sum_{G\in A}r(G)=\sum_{G\in B}r(G).
\end{equation}

To count the number of simple switches that transform a graph in $A$ to a graph in $B$ we need to find all ordered pairs of vertices $(x,y)$ such that $\{x,y\}\in E(G)$ and $\{u,x\},\{v,y\}\notin E$ as this will allow us to perform the switch $G'=G-\{u,v\}-\{x,y\}+\{u,x\}+\{v,y\}$ where the resulting graph is in $B$. So for every $x\notin N_G(u)\cup\{u\}$ we have $d_x-|N_G(x)\cap N_G(v)|-\IV{\{x,v\}\in E(G)}$ options to choose the vertex $y$ (where $\IV{\varphi}$ denotes the indicator variable of the event $\varphi$).
$$\forall G\in A,\qquad r(G)=\sum_{x\notin N_G(u)\cup\{u\}} (d_x - |N_G(x)\cap N_G(v)| - \IV{\{x,v\}\in E(G)});$$
Similarly, for counting the number of simple switches that transform a graph in $B$ to a graph in $A$ we need to find all ordered pairs of vertices $(x,y)$ such that $\{u,x\},\{v,y\}\in E(G)$ but $\{x,y\}\notin E$ as this will allow us to perform the switch $G'=G-\{u,x\}-\{v,y\}+\{u,v\}+\{x,y\}$ where the resulting graph is in $A$. So for every $x\in N_G(u)\cup\{u\}$ we have $d_v-|N_G(x)\cap N_G(v)|-\IV{\{x,v\}\in E(G)}$ options to choose the vertex $y$.
$$\forall G\in B,\qquad r(G)=\sum_{x\in N_G(u)} (d_v - |N_G(x)\cap N_G(v)| - \IV{\{x,v\}\in E(G)}).$$

To prove the upper bound of \eqref{eq:edgeprobGNDbar} we establishing a lower bound on the left hand side of \eqref{e:switchgraphdeg},
\begin{eqnarray*}
\sum_{G\in A}r(G) & = &  \sum_{G\in A}\sum_{x\notin N_G(u)\cup\{u\}}(d_x - |N_G(x)\cap N_G(v)| - \IV{\{x,v\}\in E(G)})\\
& \geq &\sum_{G\in A}\left(\sum_{x\in V}d_x - d_u-\sum_{x\in N_G(u)}d_x - \sum_{x\in V}|N_G(x)\cap N_G(v)|- \sum_{x\in V}\IV{\{x,v\}\in E(G)}\right).
\end{eqnarray*}
All the above summations are bounded as follows. $\sum_{x\in V}d_x=\od n$; $\sum_{x\in N_G(u)}d_x\leq D\cdot d_u$; $\sum_{x\in V}|N_G(x)\cap N_G(v)|=\sum_{x\in N_G(v)}d_x\leq D\cdot d_v$; $\sum_{x\in V}\IV{\{x,v\}\in E(G)}=d_v$. Putting it altogether yields the following lower bound on $\sum_{G\in A}r(G)$.
$$\sum_{G\in A}r(G)\geq|A|(\od n-(D+1)(d_u+d_v)).$$
On the other hand we have that the right hand side of \eqref{e:switchgraphdeg} satisfies
$$\sum_{G\in B}r(G)\leq|B|d_u\cdot d_v.$$
Putting the two together implies $\frac{|A|}{|B|}\leq\frac{d_ud_v}{\od n-(D+1)(d_u+d_v)}$, and therefore $\frac{|A|}{|A|+|B|}\leq\frac{d_ud_v}{\od n +d_u\cdot d_v-(D+1)(d_u+d_v)}$.

Let $\mathcal{G}'$ denote the family of graphs with the given degree sequence $\mathbf{d}$ such that the assertion of Theorem \ref{t:GND_const_density} holds, then $\Prob{\GNd{\bd}\in\mathcal{G}'}=1-o(1)$. We note that if $G\in\mathcal{G}'$ then any two non-adjacent vertices can have at most two common neighbors, and any two adjacent vertices can have at most one common neighbor.

Denote by $A'=A\cap \mathcal{G}'$ and by $B'=B\cap \mathcal{G}'$. By the upper bound just proved and our assumption that $D=O(1)$, we have that $\Prob{\GNd{\bd}\in A'}\leq\Prob{\GNd{\bd}\in A}=o(1)$. This clearly implies that $\Prob{\GNd{\bd}\in B'}= (1-o(1))$ and in particular $|B'|\geq(1-o(1))|B|$.

To get the lower bound of \eqref{eq:edgeprobGNDbar}, we upper bound the left hand side of \eqref{e:switchgraphdeg}
\begin{eqnarray*}
\sum_{G\in A}r(G) & = &  \sum_{G\in A}\sum_{x\notin N_G(u)\cup\{u\}}
(d_x - |N_G(x)\cap N_G(v)| - \IV{\{x,v\}\in E(G)})\\
& \leq &\sum_{G\in A}\sum_{x\neq u,v}d_x =  |A|(\od n - d_u-d_v),
\end{eqnarray*}
and lower bound of the right hand side of \eqref{e:switchgraphdeg} by going over only the graphs in $B'$,
\begin{eqnarray*}
\sum_{G\in B}r(G) & \geq & \sum_{G\in B'}r(G) =  \sum_{G\in B'}\sum_{x\in N_G(u)}(d_v - |N_G(x)\cap N_G(v)| - \IV{\{x,v\}\in E(G)})\\
& \geq & |B'| d_u\cdot (d_v - 2).
\end{eqnarray*}
Since $u$ and $v$ can be interchanged we can similarly infer that $\sum_{G\in B}r(G)\geq |B'|d_v\cdot (d_u-2)$. Averaging the last two inequalities implies $\sum_{G\in B}r(G)\geq |B'|(d_ud_v-d_u-d_v)$. Recalling that $|B|=(1+o(1))|B'|$ and plugging in the above results in $\frac{|A|}{|B|}\geq
(1-o(1))\frac{|A|}{|B'|}\geq(1-o(1))\frac{d_ud_v - d_u - d_v}{\od n - d_u - d_v}$. Finally, the lower of \eqref{eq:edgeprobGNDbar} $\frac{|A|}{|A+|B|}\geq(1-o(1))\frac{d_ud_v - d_u - d_v}{\od n +d_ud_v- 2d_u - 2d_v}$ follows.
\end{proof}

As a corollary of Proposition \ref{p:edgeprobGNDbar} we upper bound the probability that a predetermined set of edges is contained in a random regular graph. It should be noted that a similar result already appears in \cite{KimSudVu2007}, but this result applies only when the set of edges is of constant cardinality which will not be sufficient for our purposes.
\begin{cor}\label{c:E0contained}
For every fixed positive $\varepsilon>0$ and fixed integer $d\geq3$ there exists a constant $C=C(\varepsilon, d)$ such that if $V$ is a fixed set of $n$ vertices and $E_0\subseteq \binom{V}{2}$ is a set of $m\leq (1-\varepsilon)\frac{nd}{2}$ pairs of vertices from $V$, then
$$\Prob{E_0\subseteq E(\GND)}\leq \left(\frac{Cd}{n}\right)^m.$$
\end{cor}
\begin{proof}
Let $G\sim \GND$ and assign some arbitrary ordering on the pairs of $E_0=\{e_1,e_2,\ldots,e_m\}$. Let $F=(V,E_0)$ be the graph composed of the edges of $E_0$ with $V$ as vertex set. For every $1\leq i\leq m$ denote by $F_i=(V,\{e_1,\ldots,e_i\})$ and by $F_0$ the empty graph on $V$. Note that if $\Delta(F)>d$, then the claim is trivially true hence we can and will assume $\Delta(F)\leq d$. We bound the event that all pairs in $E_0$ are in $G$ by bounding the probability that $e_i\in E(G)$ conditioned on the event that the previous edges $\{e_1,\ldots,e_{i-1}\}$ were selected in the random graph, $G$.
$$\Prob{E_0\subseteq E(G)}=\prod_{i=1}^m\cProb{e_i\in E(G)}{E(F_{i-1})\subseteq E(G)}.$$
For $1\leq i\leq m$ denote by $\bd_i$ the degree sequence $\{d-d_{F_{i-1}}(v)\}_{v\in V}$. The maximum degree in $\bd_i$ is clearly bounded by $d=O(1)$, and the degree sum of $\bd_i$ is $dn-2(i-1)>dn-2m\geq\varepsilon dn$. Let $G_i\sim \GNd{\bd_i}$, then
$$\cProb{e_i\in E(G)}{E(F_{i-1})\subseteq E(G)}=\cProb{e_i\in E(G_i)}{E(G_i)\cap E(F_{i-1})=\emptyset}\leq\frac{\Prob{e_i\in E(G_i)}}{\Prob{E(G_i)\cap E(F_{i-1})=\emptyset}}.$$
Using Proposition \ref{p:edgeprobGNDbar} we have that $\Prob{e_i\in E(G_i)}\leq\frac{d^2}{\varepsilon dn-2d(d+1)}\leq\frac{d}{2\varepsilon n}$. To lower bound the denominator $\Prob{E(G_i)\cap E(F_{i-1})=\emptyset}$ we resort to Theorem \ref{t:McK85}. Let $\gamma$ and $\nu$ be as defined in Theorem \ref{t:McK85} then $\gamma=\frac{1}{dn-2(i-1)}\sum_{v\in V}\binom{d-d_{F_{i-1}}(v)}{2}\leq\frac{1}{\varepsilon dn}\cdot n\cdot\binom{d}{2}=\frac{d-1}{2\varepsilon}$ and $\nu=\frac{1}{dn-2(i-1)}\sum_{uv\in F_{i-1}}(d-d_{F_{i-1}}(u))(d-d_{F_{i-1}}(v))\leq\frac{1}{\varepsilon dn}\cdot(i-1)d^2\leq\frac{dm}{\varepsilon n}\leq\frac{d^2(1-\varepsilon)}{2\varepsilon}$. Plugging it in we have $\Prob{G_i\cap F_{i-1}=\emptyset}\geq(1-o(1))\cdot\exp\left(-\gamma^2-\gamma-\nu+o(1)\right)\geq C'(\varepsilon,d)$ where $C'$ is a constant that depends on $\varepsilon$ and $d$. The claim follows from putting together both bounds.
\end{proof}
The following is a well-known concentration result for $\GNd{\bd}$ which makes use of martingales and the Azuma-Hoeffding inequality (see e.g \cite{AloSpe2008}, \cite{McD98}).
\begin{thm}[\cite{Wor99}]\label{t:Azuma4RegGraphs}
For every graphic degree sequence $\bd=\{d_v\}_{v\in V}$ with $D=O(1)$ and positive constant $c>0$ , if $X$ is a random variable defined on $\GNd{\bd}$ such that $|X(G)-X(G')|\leq c$ for every pair of graphs $G\sim G'$, then for all $\varepsilon>0$
\begin{equation*}
\Prob{{X\leq (1-\varepsilon)\Exp{X}}}\leq \exp{\left(-\frac{\varepsilon^2\Exp{X}^2}{\od nc^2}+\gamma(\gamma+1)+o(1)\right)}\leq \exp{\left(-\frac{\varepsilon^2\Exp{X}^2}{Dnc^2}+\frac{D^2-1}{4}+o(1)\right)},
\end{equation*}
where $\od=\frac{1}{n}\sum_v d_v$ and $\gamma = \frac{1}{\od n}\sum_v\binom{d_v}{2}=\frac{1}{2}\left(\frac{\sum_v d_v^2}{\sum_v d_v} -1\right)\leq \frac{D-1}{2}$.
\end{thm}
\begin{rem}
Theorem \ref{t:Azuma4RegGraphs} appears in \cite{Wor99} as a concentration result for the random regular graph model $\GND$, but the proof of this more general result can be derived using the exact same arguments and plugging in the probability of the event of generating a simple graph in the configuration model with the given degree sequence $\bd$ (see e.g. \cite{McK85}) instead of the probability of this event for regular case.
\end{rem}

\section{Connectivity and Perfect Matching}\label{s:ConnAndPerfMatch}
In this section we proceed to prove Theorems \ref{t:edge_conn}, \ref{t:vert_conn} and \ref{t:perf_match}, where our main technical ingredient will be the Lov\`{a}sz Local Lemma.
\subsection{Edge and vertex connectivity}
For every integer $k\geq 1$, a graph $G=(V,E)$ is \emph{$k$-edge connected} if the removal of any $k-1$ edges from $G$ does not result in a disconnected graph, or alternatively, if there is no partition of the vertex set $V=V_1\cup V_2$ satisfying $e(V_1,V_2)<k$. Similarly, $G$ is \emph{$k$-vertex connected} if the removal of every $k-1$ vertices does not result in a disconnected graph, or equivalently, if every subset of vertices $U$ of
cardinality at most $|V|/2$ satisfies $|N_G(U)|\geq k$. Note that for every graph $G$ if $k'<k$ then
\begin{eqnarray}
r_\ell(G, \EdgeConn{k}) &\leq& r_\ell(G, \EdgeConn{k'});\label{e:edgeconn_monotonicity}\\
r_\ell(G, \VertConn{k}) &\leq& r_\ell(G, \VertConn{k'}).\label{e:vertconn_monotonicity}
\end{eqnarray}

Clearly, if $G$ is $d$-regular, then the removal of $d-k+1$ edges incident to the same vertex results in  a graph that is neither $k$-edge connected nor $k$-vertex connected, and hence the following trivial upper bound for the local resilience of $d$-regular graphs with respect to being $k$-edge connected and $k$-vertex connected is established.
\begin{clm}\label{c:res_kconn_triv_uppboud}
For every pair of integers $3\leq d\leq n-1$ and $1\leq k\leq d$, and every $d$-regular graph $G$,
\begin{eqnarray*}
r_\ell(G,\EdgeConn{k})&\leq&d-k+1;\\
r_\ell(G,\VertConn{k})&\leq&d-k+1.
\end{eqnarray*}
\end{clm}
Applying the Lov\`{a}sz Local Lemma produces a different upper bound on the local resilience of any $d$-regular graph with respect to being $k$-edge connected or $k$-vertex connected for every integer $k\geq 1$.
\begin{prop}\label{p:res_kconn_uppbound}
For every integer $3\leq d\leq n-1$, if $G=(V,E)$ is a $d$-regular graph on $n$ vertices, then there exists a subgraph $H$ with $\Delta(H)\leq d/2+4\sqrt{d\ln d}$ such that the graph $G-H$ is disconnected.
\end{prop}
\begin{proof}
Partition $V$ into two sets $V_1$ and $V_2$ by choosing for each vertex uniformly at random a side. For every vertex $v\in V$ the random variable $d_{G_{V_1,V_2}}(v)$ is distributed according to the Binomial distribution with $d$ trials and success probability $1/2$. Let $A_v$ denote the event $d_{G_{V_1,V_2}}(v)>d/2+4\sqrt{d\ln d}$, then setting $\Delta=8\sqrt{\ln d/d}\leq 5$ in Theorem \ref{t:Chernoff} item \ref{i:Chernoff1} and using the fact that $(\Delta+1)\ln(\Delta+1)-\Delta>\Delta^2/10$ $\Prob{A_v}<\exp(-\frac{d}{2}\cdot\frac{\varepsilon^2}{10})\leq d^{-3}$. If $u$ is a vertex of distance at least $3$ from $v$, then clearly the events $A_v$ and $A_u$ are independent, hence $A_v$ is dependent of less than $d^2$ other such events. For $d\geq 3$ we have that $e\cdot d^2\cdot d^{-3}<1$, and hence the Lov\`{a}sz Local Lemma (see e.g. \cite[Corollary 5.1.2]{AloSpe2008}) asserts that there exists a partition of $V$ into $V_1$ and $V_2$ such that $d_{G_{V_1,V_2}}(v)\leq d/2+4\sqrt{d\ln d}$ for every vertex $v\in V$. Taking $H=G_{V_1,V_2}$ completes the proof.
\end{proof}
Proposition \ref{p:res_kconn_uppbound} therefore implies that for every pair of integers $3\leq d\leq n-1$ and $1\leq k\leq d$ we have
\begin{eqnarray}
r_\ell(G,\EdgeConn{k})&\leq& d/2+4\sqrt{d\ln d};\\
r_\ell(G,\VertConn{k})&\leq& d/2+4\sqrt{d\ln d}.
\end{eqnarray}
Note that in light of Claim \ref{c:res_kconn_triv_uppboud}, Proposition \ref{p:res_kconn_uppbound} provides an improved upper bound only for $k\leq d/2-4\sqrt{d\ln d}+1$.

To get a lower bound on the local resilience with respect to $k$-edge connectivity, we turn to $(n,d,\lambda)$ pseudo-random graphs.
\begin{prop}\label{p:res_kedgeconn_lowbound}
For every integer $3\leq d\leq n-1$, if $G=(V,E)$ is an $(n,d,\lambda)$-graph with $\lambda\leq d\cdot\frac{n-4}{3n-4}$, then for every subgraph $H\subseteq G$ satisfying $\Delta(H)\leq\frac{d-\lambda}{2}(1-\frac{4}{n})$, the subgraph $G-H$ is $(d-\Delta(H))$-edge connected.
\end{prop}
\begin{proof}
Let $V=V_1\cup V_2$ be some partition of the vertex set of $G$ where $|V_1|\leq |V_2|$, and denote by $\Delta=\Delta(H)$. Corollary \ref{c:expmixlem2} implies that $e_{G-H}(V_1, V_2)\geq e_{G}(V_1,V_2)- |V_1|\cdot\Delta\geq f_{G,\Delta}(|V_1|)$, where $f_{G,\Delta}:[\lfloor n/2\rfloor]\rightarrow \mathbb{R}$ is a function defined by $f_{G,\Delta}(1)=d-\Delta$ and $f_{G,\Delta}(t)=t(d-\lambda-\Delta)-t^2\cdot\frac{d-\lambda}{n}$ for every $2\leq t\leq n/2$. Note that with the assumptions on $\lambda$ and $\Delta$, and by using standard tools to analyze the extrema of $f_{G,\Delta}$, the function $f_{G,\Delta}$ attains its minimum at $t=1$, and hence $e_{G-H}(V_1, V_2)\geq f_{G,\Delta}(1)=d-\Delta$ which completes the proof.
\end{proof}

Using the bound on the typical second eigenvalue of $\GND$ for fixed values of $d$ given by Theorem \ref{t:Fri2008} we get the following corollary.
\begin{cor}\label{c:res_kedgeconn_lowbound}
For every fixed $d\geq 3$ and for every $\frac{d}{2}+\sqrt{d}\leq k\leq d$, if $G\sim\GND$, then w.h.p. for every subgraph $H\subseteq G$ satisfying $\Delta(H)\leq d-k$, the subgraph $G-H$ is $k$-edge connected.
\end{cor}

We proceed to explore the local resilience of an $(n,d,\lambda)$-graph with respect to vertex connectivity. To this end, will state our result under some additional assumptions on the graph. Specifically, we require the graph to be locally ``sparse''. Although this constraint may seem somewhat artificial, it arises naturally in the setting of random $d$-regular graphs as stated in Theorem \ref{t:GND_const_density}.
\begin{prop}\label{p:res_kvertconn_lowbound}
There exists an integer $d_0>0$ such that for every $\varepsilon>0$ and integer $d_0\leq d\leq n-1$, if $G$ is an $(n,d,\lambda)$-graph satisfying $\rho(G,d+d^2/\varepsilon)\leq 1$, then for every subgraph $H\subseteq G$ satisfying $\Delta(H)\leq\frac{d-\lambda}{2}-\varepsilon$, the subgraph $G-H$ is $(d-\Delta(H))$-vertex connected.
\end{prop}
\begin{proof}
Set $\Delta=\Delta(H)$, and let $U$ be a subset of vertices of cardinality $u\leq d-\Delta-1$. Assume $U$ is a minimal separating set in $G-H$, and let $A\subseteq V\setminus U$ be the vertex set of a smallest connected component after the removal of $U$, then $V\setminus U=A\cup B$ where $e_{H}(A,B)= e_G(A,B)$ and  $a=|A|\leq |B|=n-a-u$. By the minimality of $U$ we can also assume $U=N_{G-H}(A)$, and therefore $a\geq 2$, since the removal of at most $d-\Delta-1$ vertices cannot disconnect a single vertex from $G-H$. 

Clearly, $e_H(A,B\cup U)\leq \Delta a$ and $e_{G-H}(A,B\cup U)=e_{G-H}(A,U)\leq du$. On the other hand Corollary \ref{c:expmixlem2} implies $\frac{a(n-a)(d-\lambda)}{n}\leq e_G(A,B\cup U)\leq \Delta a+du$. Noting that $a<\frac{n}{2}$ and $\Delta\leq \frac{d-\lambda}{2}-\varepsilon$ we have that $\frac{a(n-a)(d-\lambda)}{n}>\frac{a(d-\lambda)}{2}$ and $\Delta a\leq \frac{a(d-\lambda)}{2}-\varepsilon a$. It follows that $a<\frac{du}{\varepsilon}=O(1)$.

Let $U_k=\{v\in U\;:\;d_{G-H}(v,A)=k\}$. Since $(G-H)[A]$ is connected we have $e_{G-H}(A)\geq a-1$. On the other hand, since $|A\cup U|=a+u\leq \frac{d^2}{\varepsilon}+d$, from our assumption on the density of small sets in $G$ (and hence in $G-H$) we have that $k|U_k|=e_{G-H}(A,U_k)\leq e_{G-H}(A\cup U_k)-e_{G-H}(A)\leq a+|U_k|-(a-1)=|U_k|+1$. This implies that $|U_2|\leq 1$ and that $|U_k|=0$ for every $k>2$. This assumption on $G$ also implies that $e_G(A,B\cup U)=da-2e_G(A)\geq (d-2)a$, and thus $a(d-2-\Delta)\leq e_{G-H}(A,B\cup U)=e_{G-H}(A,U)\leq d-\Delta$ since every vertex in $U$ is a neighbor of some vertex in $A$ and there is at most one with two neighbors. As $\Delta\leq\frac{d}{2}$ which in turn implies $a\leq 1+\frac{2}{d-2-\Delta}\leq 1+\frac{4}{d-4}<2$ for large enough $d$ which is a contradiction.
\end{proof}

Much like for Corollary \ref{c:res_kedgeconn_lowbound} we use the bound for a typical graph in $\GND$ of Theorem \ref{t:Fri2008} and Theorem \ref{t:GND_const_density} to guarantee that a typical graph in $\GND$ satisfies the density requirements of Proposition \ref{p:res_kvertconn_lowbound} to infer the following.
\begin{cor}\label{c:res_kvertconn_lowbound}
There exists an integer $d_0>0$ such that for every fixed $d\geq d_0$ and $\frac{d}{2}+\sqrt{d}\leq k\leq d$, if $G\sim\GND$ then w.h.p. for every subgraph $H\subseteq G$ satisfying $\Delta(H)\leq d-k$, the subgraph $G-H$ is $k$-vertex connected.
\end{cor}

We now derive Theorems \ref{t:edge_conn} and \ref{t:vert_conn} from Claim \ref{c:res_kconn_triv_uppboud}, Proposition \ref{p:res_kconn_uppbound}, and  Corollaries \ref{c:res_kedgeconn_lowbound} and \ref{p:res_kvertconn_lowbound}. Clearly, the conjunction of Claim \ref{c:res_kconn_triv_uppboud} and Corollary \ref{c:res_kedgeconn_lowbound} implies item \ref{i:edge_conn1} of Theorem \ref{t:edge_conn} and the conjunction of Claim \ref{c:res_kconn_triv_uppboud} and Corollary \ref{c:res_kvertconn_lowbound} implies item \ref{i:vert_conn1} of Theorem \ref{t:vert_conn}. Both these items state that for high enough values of $k$ the local resilience of $G$ with respect to being $k$-edge-connected or $k$-vertex-connected is \emph{exactly} $d-k+1$. Set $k_0=\frac{d}{2}+\sqrt{d}$ then using \eqref{e:edgeconn_monotonicity} and \eqref{e:vertconn_monotonicity} Corollaries \ref{c:res_kedgeconn_lowbound} and \ref{c:res_kvertconn_lowbound} also provide a lower bound of $\frac{d}{2}-\sqrt{d}$ for every $k<k_0$, hence establishing the lower bound of the other items in Theorems \ref{t:edge_conn} and \ref{t:vert_conn}. The upper bound in all these items follows from Claim \ref{c:res_kconn_triv_uppboud}, Proposition \ref{p:res_kconn_uppbound}. Both theorems demonstrate an interesting threshold phenomenon for both $k$-connectivity properties that occurs around $d/2$, as is plotted in Figure \ref{f:kconn}.

\begin{figure}
\centering
\includegraphics[width=125mm]{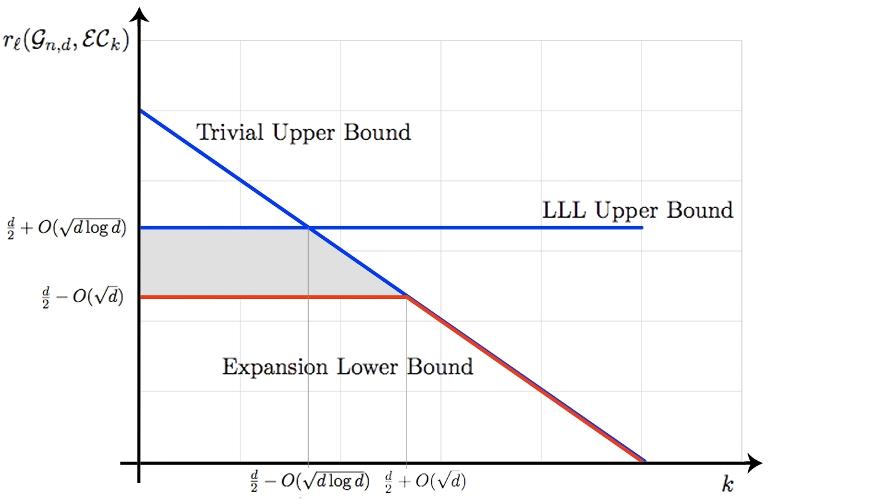}
\caption{The local resilience of edge connectivity}\label{f:kconn}
\end{figure}

\subsection{Perfect matching}
Let $G=(V,E)$ be a graph. We say that a subset of edges $M\subseteq E$ is a \emph{matching} if no two edges in $M$ share a vertex. $M$ is a \emph{perfect matching} if the edges of $M$ cover all of the vertices of $G$. Clearly, in order for $G$ to contain a perfect matching, $|V|$ must be even.

To derive the lower bound of the local resilience of a typical random $d$-regular graph with respect to containment of a perfect matching, we resort to the following lemma which states that every (not necessarily regular) graph with an even number of vertices and a large enough maximum degree can be partitioned into two equal sets such that the induced bipartite subgraph between these sets has a minimum degree not much smaller than half the minimum degree of the original graph.
\begin{lem}\label{l:good_partition}
For every graph $G=(V,E)$ on $2n$ vertices with maximum degree $\Delta(G)\geq 3$ there exists a partition of its vertex set $V=V_1\cup V_2$ such that $|V_1|=|V_2|=n$ and $\delta(G_{V_1,V_2})\geq \delta(G)/2 - 5\sqrt{\Delta(G)\ln\Delta(G)}$.
\end{lem}
\begin{proof}
Fix an arbitrary partition of the vertex set $V$ into $n$ pairs. Add the vertices of each pair to two opposing sets, $V_1$ and $V_2$, uniformly at random, and let $H=G_{V_1,V_2}$. Clearly, for every $v\in V$, the random variable $d_{H}(v)$ is Binomially distributed with expectation $\frac{d_G(v)}{2}\leq\Exp{d_{H}(v)}=\frac{d_{G}(v)}{2}+\eta_v\leq d_G(v)$, where $\eta_v\in\{0,1/2\}$ depending on whether the vertex paired with $v$ is a neighbor of $v$ or not. Set $\delta=\delta(G)$, $\Delta=\Delta(G)$, and $\zeta=\sqrt{\Delta\ln\Delta}$, then by Theorem \ref{t:Chernoff} item \ref{i:Chernoff2}
\begin{equation*}
\Prob{d_{H}(v)< \delta/2-5\zeta}\leq \exp\left(-\frac{d_{G}(v)}{4}\cdot\frac{25\zeta^2}{d_G(v)^2}\right)\leq \exp\left(-\frac{25\zeta^2}{4\Delta}\right)=\Delta^{-25/4}.
\end{equation*}
For every vertex $v$, let $A_v$ denote the event $d_{H}(v)\leq \delta/2-5\zeta$. The event $A_v$ depends on at most $2(\Delta+1)^2$ other events $A_u$ (all other vertices of distance at most 2 from $v$ or $v'$ along with their pairs, where $v'$ is the vertex paired with $v$). Recalling that $\Delta\geq 3$, we have that $e\cdot 2(\Delta+1)^2\cdot \Delta^{-25/4}<1$, and by the symmetric version of the Lov\`{a}sz Local Lemma (see e.g. \cite[Corollary 5.1.2]{AloSpe2008}) there exists a partition of the vertex set into two equal parts $V_1$ and $V_2$ such that $\delta(G_{V_1,V_2})\geq \delta(G)/2 - 5\sqrt{\Delta(G)\ln\Delta(G)}$.
\end{proof}

By using the expansion properties of $(n,d,\lambda)$-graphs, we can use Hall's criterion to deduce a lower bound of the local resilience of these graphs with respect to containment of a perfect matching.
\begin{prop}\label{p:res_perfmatch_lowbound}
If $d\geq 3$ and $G$ is an $(2n,d,\lambda)$-graph then for every subgraph $H\subseteq G$ satisfying $\Delta(H)\leq d/2-10\sqrt{d\ln d}-2\lambda$, the graph $G-H$ contains a perfect matching.
\end{prop}
\begin{proof}
Fix $G=(V,E)$ and $H\subseteq G$ as in the proposition and let $G'=G-H$, then $G'$ has $2n$ vertices and satisfies $\delta(G')\geq d/2+10\sqrt{d\ln d}+2\lambda$ and $3\leq \Delta(G')\leq d$. By Lemma \ref{l:good_partition} there exists a partition of its vertex set $V=V_1\cup V_2$ where $|V_1|=|V_2|=n$ such that $\delta'=\delta(G'_{V_1,V_2})\geq d/4+\lambda$.

Let $m=\lceil\frac{n}{2}\rceil$, fix an integer $1\leq s\leq m$ and let $S\subseteq V_1$ (without loss of generality) be a set of cardinality $s$. Denote by $T=N_{G'_{V_1,V_2}}(S)$, and assume that $|T|=t\leq s-1<\frac{n}{2}$. By Lemma \ref{l:expmixlem} we have that $s\cdot\delta'\leq e_{G'_{V_1,V_2}}(S,T)\leq e_G(S,T)\leq s\left(\frac{td}{2n}+\lambda\sqrt{\frac{t}{s}}\right)<s\left(\frac{d}{4}+\lambda\right)$, which is a contradiction. If, on the other hand $s>m$ then $t\geq m$, as otherwise there exists a subset of vertices $S'\subset S$ of cardinality $s'=m$ such that $|N_{G}(S')|\leq t < m$, contradicting the previous case. Now, let $S'=V_2\setminus N_{G}(S)$ and note that $|S'|=n-t\leq m$ and $|N_{G}(S')|\leq |V_1\setminus S|=n-s$. By the previous case, $n-t\leq n-s$, and thus $t\geq s$. The proposition follows from Hall's criterion.
\end{proof}

Applying Theorem \ref{t:Fri2008} to bound the typical value of the second eigenvalue of $\GND$ we deduce the following corollary which establishes the lower bound of Theorem \ref{t:perf_match}
\begin{cor}
For every fixed $d\geq 3$, if $G\sim\Gtwond$ then w.h.p. for every subgraph $H\subseteq G$ satisfying $\Delta(H)\leq d/2-10\sqrt{d\ln d}-4\sqrt{d}$, the graph $G-H$ contains a perfect matching.
\end{cor}

The following proposition, which is quite similar to Lemma \ref{l:good_partition}, will imply the upper bound of Theorem \ref{t:perf_match} by considering $H=G[U]$ as the subgraph removed from $G$ such that $G-H$ contains no perfect matching as it contains an independent set with more than half of the vertices.
\begin{prop}\label{p:res_perfmatch_uppbound}
There exists an integer $d_0>0$ such that for every integer $d_0\leq d\leq 2n-1$, if $G=(V,E)$ is a $d$-regular graph on $2n$ vertices, then there exists a subset $U\subseteq V$ of $n+1$ vertices such that $\Delta(G[U])\leq d/2+2\sqrt{d\ln d}+2$.
\end{prop}
\begin{proof}
Fix an arbitrary partition of the vertex set $V$ into $n$ pairs. Remove from $G$ all edges spanned by these pairs and denote the resulting graph by $G'=(V,E')$. Add the vertices uniformly at random to opposing sets, $V_1$ and $V_2$, and let $H=G'_{V_1,V_2}$. For any vertex $v\in V$, $d_{H}(v)$ is a random variable Binomially distributed with expectation $\Exp{d_{H}(v)}=d_{G'}(v)/2$. Setting $\varepsilon=\frac{4\sqrt{d\ln d}}{d_{G'}(v)}\leq 1$ for large enough $d$, then by Theorem \ref{t:Chernoff} item \ref{i:Chernoff3} and noting that $(\varepsilon+1)\ln(\varepsilon+1)-\varepsilon>\varepsilon^2/3$ we have
$$\Prob{|d_{H}(v)- d_{G'}(v)/2|> \frac{\varepsilon d_{G'}(v)}{2}}\leq 2\cdot\exp\left(-\frac{\varepsilon^2d_{G'}(v)}{6}\right)=2\exp\left(-\frac{8\ln d}{3}\right)=2\cdot d^{-8/3}.$$
For every vertex $v$, let $A_v$ denote the event $|d_{H}(v)- d_{G'}(v)/2|> \frac{\varepsilon d_{G'}(v)}{2}$. There exists a positive constant $C$ such that this event depends on at most $C\cdot d^2$ other events $A_u$ (all other vertices of distance at most 2 from $v$, where the original pairs chosen are considered as edges).As $e\cdot Cd^2\cdot 2d^{-8/3}<1$ for large enough $d$ by the symmetric version of the Lov\`{a}sz Local Lemma (see e.g. \cite[Corollary 5.1.2]{AloSpe2008}) there exists a partition of the vertex set into two equal parts $V_1$ and $V_2$ satisfying such that for every $v\in V$, $|d_{H}(v)-\frac{d}{2}|\leq 2\sqrt{d\ln d}+1$. Now, fix some vertex $v\in V_2$ and set $U=V_1\cup \{v\}$ implying $\Delta(G[U])\leq\max\{\Delta(G[V_1])+1,d-\Delta(G_{V_1,V_2})\}\leq \frac{d}{2}+2\sqrt{d\ln d}+2$ as claimed.
\end{proof}

\section{Resilience of Hamiltonicity}\label{s:Hamiltonicity}
In this section we proceed to prove a lower bound on the local resilience of random graphs with respect to being Hamiltonian. We start with a simple proposition, which will lie in the heart of the proof of Theorem \ref{t:HamResGnd} (and that of Theorem \ref{t:HamResGnp} as well); it gives the motivation to some of the succeeding computations in this section.
\begin{defn}\label{d:booster}
For every graph $G$ we say that a non-edge $\{u,v\}\notin E(G)$ is a \emph{booster} with respect to $G$ if $G+\{u,v\}$ is Hamiltonian or $\ell(G+\{u,v\})>\ell(G)$. Moreover, for any vertex $v\in V$ we denote by
\begin{equation}
B_G(v)=\{w\notin N_G(v)\cup\{v\}\;:\;\{v,w\}\hbox{ is a booster}\}.
\end{equation}
\end{defn}

\begin{lem}\label{l:boosters2Hamiltonian}
Let $r\geq 1$ and let $G_0$ and $G_1$ be two graphs on the same vertex set $V$ of cardinality $|V|=n$ such that for every $E'\subseteq E(G_1)$ of cardinality $|E'|\leq n$ there exists a vertex $v\in V$ satisfying $|N_{G_1}(v)\cap B_{G_0\cup E'}(v)|>r$. Then for every $H\subseteq G_1$ of maximum degree $\Delta(H)\leq r$, the graph $G_0 + (G_1 - H)$ is Hamiltonian.
\end{lem}
\begin{proof}
Fix a subgraph $H\subseteq G_1$ satisfying $\Delta(H)\leq r$, and denote by $G'$ the graph $G_1-H$. We will prove that there exists an edge set of $G'$ such that its addition to $G_0$ creates a Hamiltonian graph. Start with $E'_0=\emptyset$. Assume that $E'_i$ is a subset of $i$ edges of $E(G')$. If the graph $G_0\cup E'_i$ is Hamiltonian we are done. Otherwise, by the assumption of the lemma, there exists a vertex $v_i\in V$ such that $|N_{G_1}(v_i)\cap B_{G_0\cup E'_i}(v_i)|>r$, and hence there exists at least one neighbor of $v_i$ in $G_1$, which we denote by $w_i$, such that the pair $\{v_i,w_i\}$ is still an edge in $G'$, and is a booster with respect to $G_0\cup E'_i$. It follows that either the graph $G_0+E'_i+\{v_i,w_i\}$ is Hamiltonian or $\ell(G_0+E'_i+\{v_i,w_i\})>\ell(G_0+E'_i)$. Finally, set $E'_{i+1}=E'_i\cup\{\{v_i,w_i\}\}$. Note that there must exist an integer $i_0\leq n$ such that $G_0+E'_{i_0}$ is Hamiltonian, as the length of a longest path on the vertex set of $|V|$ is at most $n-1$.
\end{proof}
\begin{rem}\label{r:boosters2Hamiltonian}
The above is actually a local resilience statement. Fix a graph $G_0$ and denote by $\Hamiltonicity_{G_0}$ the property of a graph being Hamiltonian when the edges of $G_0$ are added to it. Lemma \ref{l:boosters2Hamiltonian} states that for every $r\geq 1$ if $G_1$ is such that for every $E'\subseteq E(G_1)$ of cardinality $|E'|\leq n$ there exists a vertex $v\in V$ satisfying $|N_{G_1}(v)\cap B_{G_0\cup E'}(v)|>r$, then $r_\ell(G_1,\Hamiltonicity_{G_0})>r$.
\end{rem}

\subsection{P\'{o}sa's rotation-extension technique}\label{ss:Posa}
In this subsection we describe and apply a crucial technical tool, originally developed by P\'{o}sa \cite{Pos76}, which lies in the foundation of many Hamiltonicity results of random and pseudo-random graphs. This technique, which has come to be known as P\'{o}sa's rotation-extension, relies on the following basic operation on a longest path in a graph. We use the following two definitions.
\begin{defn}
Let $G=(V,E)$ be a graph, and let $P=(v_0,v_1,\ldots,v_t)$ be a longest path in $G$. If $\{v_i,v_t\}\in E$ for some $0\leq i\leq t-2$, then an \emph{elementary rotation} of $P$ along $\{v_i,v_t\}$ is the construction of a new longest path $P'=P-\{v_i,v_{i+1}\}+\{v_i,v_t\}=(v_0,v_1,\ldots,v_i,v_t,v_{t-1},\ldots,v_{i+1})$. We say that the edge $\{v_i,v_{i+1}\}$ is \emph{broken} by this rotation.
\end{defn}

\renewcommand{\labelenumi}{\textbf{(Q\arabic{enumi})}}
\begin{defn}\label{d:neps-expander}
We say that a graph $G=(V,E)$ on $n$ vertices is an \emph{$(n,\varepsilon)$-expander} if
\begin{enumerate}
\item\label{i:neps-expander1} Every $V_0\subseteq V$ of cardinality $|V_0|< \varepsilon n$ satisfies $|N_{G}(V_0)|\geq 10 |V_0|$;
\item\label{i:neps-expander2} Every $V_0\subseteq V$ of cardinality $\varepsilon n \leq |V_0|\leq 2\varepsilon n$ satisfies $|N_{G}(V_0)|\geq (1+12\varepsilon)n/2$.
\end{enumerate}
\end{defn}
\begin{rem}\label{r:nepsexpander}
We note that if $G=(V,E)$ is an $(n,\varepsilon)$-expander, then every $H=(V,F)$ for $F\supseteq E$ is also an $(n,\varepsilon)$-expander.
\end{rem}

Moreover, it is immediate to see that any $(n,\varepsilon)$-expander is connected.
\begin{clm}
If $G=(V,E)$ is an $(n,\varepsilon)$-expander for some $\varepsilon>0$, then $G$ is connected.
\end{clm}
\begin{proof}
Assume otherwise and let $U\subseteq V$ be a connected component of cardinality $|U|\leq n/2$. Properties \textbf{Q\ref{i:neps-expander1}} and \textbf{Q\ref{i:neps-expander2}} imply that every subset of vertices of at most $2\varepsilon n$ vertices has a non-empty neighbor set, hence we can further assume that $|U|>2\varepsilon n$. Let $U'\subseteq U$ be of cardinality $\lceil\varepsilon n\rceil$, then by property \textbf{Q\ref{i:neps-expander2}} of $G$ $|N_G(U')| > n/2$ and hence cannot be contained in $U$, a contradiction.
\end{proof}
Using these elementary rotations we proceed to show that any $(n,\varepsilon)$-expander, $G$, must be Hamiltonian or that the subset of vertices $v$ with ``large'' $B_G(v)$ must also be large. Our proof uses similar ideas to those found in \cite{HefKriSza2009}.

\begin{lem}\label{l:nEpsExpander}
If $G=(V,E)$ is an $(n,\varepsilon)-expander$ for some $\varepsilon>0$, then $G$ is Hamiltonian or must satisfy $|\{v\in V\;:\; |B_G(v)|\geq n/4+\varepsilon n\}|\geq n/4+\varepsilon n$.
\end{lem}
\begin{proof}
Take a longest path $P=(v_0,\ldots,v_t)$ in $G$. Since $P$ is a longest path, $N_G(v_0)\cup N_G(v_t)\subseteq P$. Taking any $v_i\in N_G(v_t)$, we can perform an elementary rotation along $\{v_t,v_i\}$ keeping $v_0$ fixed resulting in a longest path $P'$ in $G$. For every $i\leq r=\lceil \log_2(\varepsilon n)\rceil$ let $S_i$ be a subset of all endpoints of longest paths in $G$ obtained by performing a series of $i$ elementary rotations starting from $P$ while keeping $v_0$ fixed such that at the $j^{th}$ rotation the non-$v_0$ endpoint is in $S_j$. We construct the sequence of sets $\{S_i\}_{i=0}^r$ such that $|S_i|=2^i$, and prove it inductively. $S_0=\{v_t\}$ and hence the base of the induction is satisfied trivially. Now, taking $i<r$ we assume that the hypothesis is satisfied for all $j\leq i$ and we prove it for $i+1$. First we note that by property \textbf{Q\ref{i:neps-expander1}} of $G$ $|N_{G}(S_i)|\geq 10\cdot2^i$. Let $I=\{v_k\in N_{G}(S_i)\;:\;v_{k-1},v_k,v_{k+1}\notin\bigcup_{j=0}^i S_j\}$, where each $v_k\in I$ is a candidate to be added to $S_{i+1}$. Let $v_k\in I$, and $x\in S_i$ such that $\{x,v_k\}\in E$, and denote by $Q$ the longest path from $v_0$ to $x$ obtained from $P$ by $i$ elementary rotations fixing $v_0$. By the definition of $I$, none of $\{v_{k-1},v_k,v_{k+1}\}$ is an endpoint of one of the sequence of longest paths starting from $P$ and yielding $Q$, hence both edges $\{v_{k-1},v_k\}$ and $\{v_k,v_{k+1}\}$ were not broken and are therefore present in $Q$. Rotating $Q$ along the edge $\{x,v_k\}$ will make one of the vertices $\{v_{k-1},v_{k+1}\}$ an endpoint in the resulting path, so we assume w.l.o.g. that it is $v_{k-1}$, and hence add it to the set $S'_{i+1}$. Note that the vertex $v_{k-1}$ can also be added to the set $S'_{i+1}$ if the vertex $v_{k-2}$ in $I$, therefore
$$|S'_{i+1}|\geq \frac{1}{2}|I|\geq \frac{1}{2}\left(|N_G(S_{i})|-3\sum_{j=0}^i|S_j|\right)\geq \frac{1}{2}\left(10\cdot 2^i-3(2^{i+1}-1)\right)=2^{i+1}+\frac{3}{2}.$$ 
We set $S_{i+1}$ to be any subset of $S'_{i+1}$ of $2^{i+1}$ vertices.

We construct similarly the set $S'_{r+1}$, where this time we note that $\varepsilon n\leq |S_r|<2 \varepsilon n$, hence by property \textbf{Q\ref{i:neps-expander2}} of $G$ $|N_G(S_r)|\geq (1+12\varepsilon)n/2$.
$$|S'_{r+1}| \geq \frac{1}{2}\left(|N_G(S_{r})|-3\sum_{j=0}^r|S_j|\right)\geq \frac{1}{2}\left(\frac{n}{2}(1+12\varepsilon)-6\varepsilon n\right)=\frac{n}{4}.$$
Let $\hat{S}=S_r\cup S'_{r+1}$, then as $S_r$ and $S'_{r+1}$ are disjoint, we have that $|\hat{S}|\geq\frac{n}{4}+\varepsilon n$. Assume $\hat{S}\cap N_G(v_0)\neq \emptyset$ then $G$ must contain a cycle of length $\ell(G)$. This implies that $G$ is Hamiltonian, as otherwise $\ell(G)<n$ and since $G$ is connected there is an edge emitting out of this cycle and thus creating a path of length $\ell(G)+1$ in $G$ which is a contradiction. This implies that $\hat{S}\subseteq B_G(v_0)$. Now taking any endpoint $u_0$ in $\hat{S}$ so obtained and taking a longest path $P'$ starting from $u_0$ (which must exist since all vertices of $\hat{S}$ are endpoints of longest paths starting in $v_0$) and repeat the same argument, while rotating $P'$ and keeping $u_0$ fixed. This way we obtain the desired set $|B_G(u_0)|$ of $n/4+\varepsilon n$ endpoints for every $u_0\in \hat{S}$, thus completing the proof.
\end{proof}

\subsection{Proof of Theorem \ref{t:HamResGnd}}
We can now provide the full proof of the main result of this paper, namely the proof of Theorem \ref{t:HamResGnd}. Our goal is to prove that if $G\sim \GND$ for large enough values of $d$, then the probability that there exists a subgraph $H\subseteq G$ of maximum degree $r=r(d)$ such that the graph $G-H$ is \emph{not} Hamiltonian is $o(1)$, where the optimization of $r$ to be $(1-\varepsilon)d/6$ is deferred to the end of the proof. First, note that by Theorem \ref{t:Jan95} and Corollary $\ref{c:rand_reg_graphs_union}$ it is enough to prove this claim for $G\sim (\GNd{d_1}+\GNd{d_2})$ where $d_1,d_2\geq 3$ and $d_1+d_2=d$. So, let $G_1\sim \GNd{d_1}$ and $G_2\sim \GNd{d_2}$, such that $G$ is their union as a multigraph (where the same edge can appear in $G$ twice). The probability that there exists such a subgraph $H\subseteq G_1+G_2$ is clearly upper bounded by the probability that there exist two subgraphs $H_1\subseteq G_1$ and $H_2\subseteq G_2$ both of maximum degree $\Delta(H_1),\Delta(H_2)\leq r$ such that $((G_1-H_1)+(G_2-H_2))$ is not Hamiltonian, and that is the event we prove has probability $o(1)$. As a first phase we show that w.h.p. $G_1$ is such that even after the deletion of the edges of $H_1$ the resulting graph, $G_1-H_1$, still retains some expansion properties. Next, we will resort to a ``thinning'' of the graph $G_1-H_1$. We will actually prove that not only is $G_1-H_1$ a ``good'' expander, it also contains a subgraph $\Gamma\subseteq G_1-H_1$ with a small fraction of the edges, which is a ``good'' expander (where what a ``good'' expander in both cases will be quantitatively measured). Note that the ``thinning'' of $G_1-H_1$ is a deterministic claim. The ultimate goal of this ``thinning'' claim is to enable the application of a union bound argument over a smaller set of possible graphs. We will elaborate on this matter further down the proof. Lastly, we show that w.h.p., no matter what graph $\Gamma$ we have after the ``thinning'', there are enough edges from $G_2-H_2$ that will make the graph $\Gamma+ (G_2-H_2)$ Hamiltonian. Specifically, we will start with the graph $\Gamma$ and add boosters sequentially till having reached Hamiltonicity, where we will argue that due to Lemma \ref{l:nEpsExpander} at each step the random graph $G_2$ contains a booster even after the deletion of $H_2$.

In the coming computations there will be many dependencies on the value of some arbitrarily small $\varepsilon>0$ so we start by defining two constant values, that depend solely on $\varepsilon$, which will remove some clutter in the descriptions below. So, set
\begin{equation}\label{e:defnofmuandbeta}
\mu=\mu(\varepsilon)=\varepsilon^3, \hbox{ and } \beta=\beta(\varepsilon)=\mu/160.
\end{equation}
\renewcommand{\labelenumi}{\textbf{(P\arabic{enumi})}}
\begin{defn}\label{d:quasirand}
We say that a graph $G=(V,E)$ on $n$ vertices is $(n,d,\varepsilon)$-\emph{quasi-random} if it satisfies the following properties:
\begin{enumerate}
\setcounter{enumi}{-1}
\item\label{i:quasidrand2} $d/2\leq \delta(G)\leq\Delta(G)\leq2d$;
\item\label{i:quasidrand0} Every $U\subseteq V$ of cardinality $|U|< \mu n/14$ satisfies $e_G(U)\leq\mu d|U|/14$;
\item\label{i:quasidrand4} Every two disjoint subsets $U,W\subseteq V$ where $\beta n\leq |U|< 2\beta n$ and $|W|\geq \frac{n}{2}\left(1-\frac{\varepsilon}{2}-4\beta\right)$ satisfy $e_G(U,W)\geq \frac{d(1-\varepsilon/4)}{n}|U||W| - (1-\varepsilon)\frac{d}{2}|U|$.
\end{enumerate}
\end{defn}
\begin{rem}
Although in the proof of Theorem \ref{t:HamResGnd} the requirement of \textbf{P\ref{i:quasidrand2}} can be replaced by $\Delta(G)\leq d$ (and is more natural), this less strict condition will enable to reuse this definition in the course of the proof of Theorem \ref{t:HamResGnp} of the Binomial graph model $\GNP$ where in that case $d=np$.
\end{rem}

The following lemma is a local resilience claim for $(n,d,\lambda)$-graphs with large enough spectral gap with respect to being $(n,d,\varepsilon)$-quasi-random. 
\begin{lem}\label{l:ndl_r_quasirand}
For every $0<\varepsilon \leq 1$ there exists a constant $d_0=d_0(\varepsilon)$ such that if $G=(V,E)$ an $(n,d,\lambda)$-graph for some $d\geq d_0$ with $\lambda<\mu d/28$, then for any subgraph $H\subseteq G$ of maximum degree $\Delta(H)\leq (1-\varepsilon)d/2$, the graph $G-H$ is $(n,d,\varepsilon)$-quasi-random.
\end{lem}
\begin{proof}
\textbf{P\ref{i:quasidrand2}} is satisfied trivially by our assumption on $\Delta(H)$. Let $U\subseteq V$ of cardinality $|U|\leq \mu n/14$. Corollary  \ref{c:expmixlem3} implies that $e_{G-H}(U)\leq e_G(U)\leq \frac{d}{n}\binom{|U|}{2}+\lambda|U|(1-\frac{|U|}{2n})\leq \mu d|U|/14$, and hence \textbf{P\ref{i:quasidrand0}} is satisfied. Taking two disjoint subsets $U,W\subseteq V$ of cardinality $\beta n\leq |U|< 2\beta n$ and $|W|\geq \frac{n}{2}\left(1-\frac{\varepsilon}{2}-4\beta\right)$ \textbf{P\ref{i:quasidrand4}} clearly follows from Lemma \ref{l:expmixlem} as $e_{G-H}(U,W)\geq e_G(U,W)-\Delta(H)\cdot |U|\geq \frac{d|U||W|}{n}-\frac{\lambda}{n}\sqrt{|U||W|(n-|U|)(n-|W|)}-(1-\varepsilon)\frac{d}{2}|U|\geq\frac{d|U||W|}{n}\left(1-\frac{\lambda n}{d\sqrt{|U||W|}}\right)-(1-\varepsilon)\frac{d}{2}|U|\geq\frac{d(1-\varepsilon/4)}{n}|U||W| - (1-\varepsilon)\frac{d}{2}|U|$.
\end{proof}

Lemma \ref{l:ndl_r_quasirand} in conjunction with Theorem \ref{t:Fri2008} implies that w.h.p. for every $0<\varepsilon \leq 1$ the local resilience of random $d$-regular graphs (for constant but large enough values of $d$) with respect to being $(n,d,\varepsilon)$-quasi-random is at least $(1-\varepsilon)d/2$.
\begin{cor}\label{c:randreg_r_quasirand}
For every $0< \varepsilon \leq 1$ there exists a constant $d_0=d_0(\varepsilon)$ such that if $G\sim \GND$ for some $d\geq d_0$, then w.h.p. for every subgraph $H\subseteq G$ of maximum degree $\Delta(H)\leq (1-\varepsilon)d/2$, the graph $G-H$ is $(n,d,\varepsilon)$-quasi-random.
\end{cor}
Corollary \ref{c:randreg_r_quasirand} implies that if we take 
\begin{equation}\label{e:rbound1}
r\leq (1-\varepsilon)d_1/2
\end{equation}
for some $0<\varepsilon\leq 1$, then w.h.p. for every $H_1\subseteq G_1$ of maximum degree $\Delta(H_1)\leq r$ the graph $G_1-H_1$ is $(n,d_1,\varepsilon)$-quasi-random. If $G_1$ is such a graph that the above is not satisfied we say that $G_1$ is \emph{corrupted}. Assume $G_1$ is not corrupted and fix a subgraph $H_1\subseteq G_1$ as above, then $G'=G_1-H_1$ is an $(n,d_1,\varepsilon)$-quasi-random graph. We proceed to show that every $(n,d_1,\varepsilon)$-quasi-random graph contains a subgraph which retains strong expansion properties but with an arbitrarily small constant fraction of the edges.

Before proceeding to the next phase, which we call the ``thinning'' of $G'$, we describe a natural attempt one may try to prove the theorem. Lemma \ref{l:boosters2Hamiltonian} can be used to ``eliminate'' the need to consider $H_2$ explicitly, and instead bound the probability that there exists a subgraph $H_1\subseteq G_1$ of maximum degree $r$ and a set $E_0\subseteq E(G_2)$ of cardinality $|E_0|\leq n$ for which every vertex $v\in V$ satisfies $|N_{G_2}(v)\cap B_{(G_1-H_1)\cup E_0}(v)|\leq r$. Then, one may try to use the union bound by going over all possibilities for $H_1$, that is, graphs of maximum degree $r$, on the vertex set $V$, and then bounding the probability of the aforementioned event by conditioning on $H_1\subseteq G_1$. Although for every such $H_1$ the probability is exponentially small, there are also exponentially many such graphs, implying a very weak lower bound on $r$. In what follows, we try and improve on this idea so that we can ``boost'' $r$ up to almost $d/6$. Basically, we need to do a ``union bound'' argument over a much smaller set of graphs.

\begin{prop}\label{p:GminHcontainsGamma}
For every $0<\varepsilon\leq 1$ there exists an integer $d_0(\varepsilon)>0$ such that if $G'$ is an $(n,d_1,\varepsilon)$-quasi-random graph for some $d_1\geq d_0$ then $G'$ contains an $(n,\beta)$-expander subgraph $\Gamma$ satisfying $\mu nd_1/2\leq e(\Gamma)\leq \mu nd_1$ edges.
\end{prop}
\begin{proof}
Let $G'=(V,E)$ be an $(n,d_1,\varepsilon)$-quasi-random graph, and denote by $\Gamma=(V,E')$ the random subgraph of $G'$ generated by selecting for every $v\in V$ independently and uniformly at random a set $E_v$ of $\mu d_1$ incident edges to $v$ (and leaving a single copy of a edge if it is chosen by both of its endpoints). Clearly
\begin{equation}\label{e:Gamma_edges}
\mu nd_1/2\leq e(\Gamma)\leq \mu nd_1.
\end{equation}

To prove the proposition it is enough to show that $\Gamma$ is an $(n,\beta)$-expander with positive probability. We will in fact show that it is such w.h.p..

We start by proving that $\Gamma$ satisfies \textbf{Q\ref{i:neps-expander1}}, namely that every $U\subseteq V$ of cardinality $|U|<\beta n$ satisfies $|N_{\Gamma}(U)|\geq 10|U|$. Let $U\subseteq V$ be some subset of cardinality $|U|< \beta n$, and assume that $|N_\Gamma(U)|<10|U|$. By \textbf{P\ref{i:quasidrand0}} it follows that $e_{G'}(U)\leq \mu d_1|U|/14$, and therefore $e_\Gamma(U,N_\Gamma(U))\geq\mu d_1|U|-2e_\Gamma(U)\geq\mu d_1|U|-2e_{G'}(U)\geq 6\mu d_1|U|/7$. On the other hand, set $W=U\cup N_\Gamma(U)$, then $|W|<11|U|<\mu n/14$. By \textbf{P\ref{i:quasidrand0}} $e_\Gamma(U,N_\Gamma(U))\leq e_{G'}(W)< 11\mu d_1|U|/14$, which is a contradiction.

We proceed to show that w.h.p. \textbf{Q\ref{i:neps-expander2}} is satisfied. Fix a subset of vertices $U\subseteq V$ of cardinality $\beta n\leq |U|< 2\beta n $, set $t=\frac{d_1\beta}{2}\left(\frac{\varepsilon}{4}-4\beta\right)$ and let $Z=\{v\in V\setminus U\;:\;d_{G'}(v,U)<t\}$. Assume $|Z|>\frac{n}{2}\left(1-\frac{\varepsilon}{2}-4\beta\right)$, then on one hand, $e_{G'}(U,Z)<t|Z|\leq\frac{|U|t}{\beta}$, and on the other, by \textbf{P\ref{i:quasidrand4}} it follows that $e_{G'}(U,Z)\geq \frac{d_1(1-\varepsilon/4)}{n}|Z||U|-(1-\varepsilon)\frac{d_1}{2}|U|\geq\frac{d_1|U|}{2}\left(\frac{\varepsilon}{4}-4\beta\right)=\frac{|U|t}{\beta}$, a contradiction. It therefore follows that
\begin{equation}\label{eq:zbound}
|Z|+|U|\leq n/2(1-\varepsilon/2).
\end{equation}
Set $W=V\setminus (Z\cup U)=\{v\in V\setminus U\;:\;d_{G'}(v,U)\geq t\}$, and note that by \eqref{eq:zbound} $|W|\geq n/2(1+\varepsilon/2)$. For every $w\in W$ let $A_w$ be the event that $E_w$ has no endpoint in $U$. Recalling that $d_{G'}(w)\leq 2d_1$ and that $d_1$ is large enough, property \textbf{P\ref{i:quasidrand2}} of $G'$ implies
$$\Prob{A_w}= \frac{\binom{d_{G'}(w)-d_{G'}(w,U)}{\mu d_1}}{\binom{d_{G'}(w)}{\mu d_1}} \leq \left(\frac{d_{G'}(w)-\mu d_1}{d_{G'}(w)}\right)^{t}\leq\left(1-\frac{\mu}{2}\right)^{t}\leq\exp\left(-\frac{\mu t}{2}\right)\leq\varepsilon/100,$$ 
where the first inequality follows from the fact that $\binom{a}{c}\geq\binom{a-b}{c}\cdot(\frac{a}{a-c})^b$ (see e.g. \cite[Chapter I.1]{Bol2001} for every three integers $a\geq b$ and $c\leq a-b$, and that $t\leq d_{G'}(w,U)$. Let $X=\{w\in W\;:\; A_w\hbox{ holds}\}$, then its cardinality $|X|=\sum_{w\in W}\IV{A_w}$ is the sum of $|W|$ independent indicator random variables, each of expectation $\Exp{\IV{A_w}}\leq\varepsilon/100$, and $|X|$ is therefore stochastically dominated by a random variable with distribution $\Bin(n,\varepsilon/100)$. It follows by Theorem \ref{t:Chernoff} item \ref{i:Chernoff1} that $\Prob{|X|>\varepsilon n/5}\leq\exp(-\frac{(20\ln20-19)\varepsilon n}{100})<\exp(-\varepsilon n/5)$.

As the graph $\Gamma$ contains the edges of $E_u$ for all $u\in U$ we have that $N_\Gamma(U)\supseteq N_\Gamma(U)\cap W\supseteq W\setminus X$, and by the above we know that with probability at least $\exp(-\varepsilon n/5)$ the set $W\setminus X$ has cardinality greater than $\frac{n}{2}(1+\frac{\varepsilon}{2})-\frac{\varepsilon n}{5}=\frac{n}{2}(1+\frac{\varepsilon}{10})>\frac{n}{2}(1+12\beta)$. Recall that $\beta=\varepsilon^3/160\leq \varepsilon/160$, then applying the union bound over all relevant subsets $U$, we can upper bound the probability that $\Gamma$ does not satisfy property \textbf{Q\ref{i:neps-expander2}} as follows:
$$\sum_{u=\beta n}^{2\beta n}\binom{n}{u}e^{-\varepsilon n/5} \leq \sum_{u=\beta n}^{2\beta n}\exp(u\ln \frac{en}{u}-\varepsilon n/5)\leq n\cdot\exp(\varepsilon n/10-\varepsilon n/5) = o(1),$$ 
which completes the proof.
\end{proof}

Returning to the context of the proof, Proposition \ref{p:GminHcontainsGamma} implies that if $G_1$ is not corrupted then for every $H_1\subseteq G_1$ satisfying $\Delta(H_1)\leq r\leq(1-\varepsilon)d_1/2$ the graph $G_1-H_1$ contains an $(n,\beta)$-expander subgraph.

Recall that in Remark \ref{r:boosters2Hamiltonian} we defined that a graph satisfies the property $\Hamiltonicity_\Gamma$ if the addition of the edge set of $\Gamma$ to the graph results in a Hamiltonian graph. We continue to the next phase of our proof, namely, showing that for every fixed $(n,\beta)$-expander graph, $\Gamma$, and every $\varepsilon>0$, if $d_2=d_2(\varepsilon)$ is large enough then $r_\ell(\GNd{d_2},\Hamiltonicity_\Gamma)>(1-\varepsilon)d_2/4$ with probability exponentially close to $1$.

\begin{lem}\label{l:AddingRegG2}
For every $0<\varepsilon\leq 1$ there exists a large enough constant $d_0=d_0(\varepsilon)$ such that if $G_2\sim\GNd{d_2}$ for some $d_2\geq d_0$ and $\Gamma$ is an $(n,\beta)$-expander on the same vertex set, then the probability there exists a set of edges $E_0\subseteq E(G_2)$ of cardinality $|E_0|\leq n$ for which every $v\in V$ satisfies $|N_{G_2}(v)\cap B_{\Gamma\cup E_0}(v)|\leq(1-\varepsilon)d_2/4$ is at most $e^{-\Theta(\varepsilon^2 nd_2)}$.
\end{lem}
\begin{proof}
We will assume that $\varepsilon$ is small enough and $d_2$ is large enough (as a function of $\varepsilon$), without giving explicit bounds on them. Fix a choice of at most $n$ pairs of vertices $E_0\subseteq \binom{V}{2}$, and set $\Gamma_2=\Gamma\cup E_0$. Our goal is to bound the probability that $G_2\sim\GNd{d_2}$ contains the edges of $E_0$ and that for every vertex $v\in V$ we have $|N_{G_2}(v)\cap B_{\Gamma_2}(v)|\leq(1-\varepsilon)d_2/4$, in which case we say that $E_0$ \emph{ruins} $G_2$. This clearly implies that it suffices to consider only choices of $E_0$ such that the graph $F=(V,E_0)$ satisfies $\Delta(F)\leq d_2$, and hence we proceed with this assumption. Furthermore, if $E_0\subseteq E(G_2)$, then the graph $\hG=G_2-F$ is distributed according to $\GNd{\bd}$ with degree sequence $\bd=\{d_2-d_F(v)\}_{v\in V}$, conditioned on the event that there are no overlapping edges with $E_0$. Recalling Remark \ref{r:nepsexpander} we have that $\Gamma_2$ is a $(n,\beta)$-expander, and by Lemma \ref{l:nEpsExpander} the set $A=\{v\in V: |B_{\Gamma_2}(v)| \geq \frac{n}{4}(1+4\beta)\}$ must satisfy $|A|\geq \frac{n}{4}(1+4\beta)$. As $|E_0|\leq n$, the set $U=\{v\in V\;:\; d_F(v)\geq \frac{4}{\beta}\}$ satisfies $|U|\leq \beta n/2$. Let $A'=A\setminus U$ then $|A'|\geq n(\frac{1}{4}+\frac{\beta}{2})$, and moreover, for every $v\in A'$ the set $B'(v)=B_{\Gamma_2}(v)\setminus U$ satisfies $|B'(v)|\geq n(\frac{1}{4}+\frac{\beta}{2})$.

For every $v\in A'$ let $X_v$ be the random variable equal to $|N_{\GNd{\bd}}(v)\cap B'(v)|$. Then by Proposition \ref{p:edgeprobGNDbar},
$$\Exp{X_v}\geq |B'(v)|\cdot(1-o(1))\frac{(d_2-4/\beta)(d_2-4/\beta-2)}{nd_2 + d_2^2-4d_2}\geq \frac{d_2}{4},$$
where the last inequality follows by taking $d_2$ to be large enough as a function of $\beta(\varepsilon)$. Now, set $X$ to be the random variable $\sum_{v\in A'}X_v$, then the difference in the value of $X$ of any two graphs in $\GNd{\bd}$ on the same vertex set that differ by a single switch can be at most $4$ (every endpoint can contribute $1$), and we can therefore apply Theorem \ref{t:Azuma4RegGraphs} to derive that 
$$\Prob{X(\GNd{\bd})\leq (1-\varepsilon)|A'|\frac{d_2}{4}}\leq e^{-\Theta(\varepsilon^2 nd_2)}.$$
Hence with probability exponentially close to $1$, the random variable $X$ does not deviate much from its expectation. We now estimate $X=X(\hG)$ for $\hG\sim \GNd{\bd}$ conditioned on the event that $\hG$ shares no edge with $F$. Recalling that $\Delta(F)\leq d_2$, Theorem \ref{t:McK85} guarantees that conditioning on this event can affect the probability that $X(\hG)$ is too small by only a constant factor (that depends on
$d_2$). It follows that
$$\cProb{X(\hG)\leq (1-\varepsilon)|A'|\frac{d_2}{4}}{E(\hG)\cap E_0=\emptyset}\leq \frac{\Prob{X(\GNd{\bd})\leq (1-\varepsilon)|A'|\frac{d_2}{4}}}{\Prob{E(\GNd{\bd})\cap E_0=\emptyset}}\leq e^{-\Theta(\varepsilon^2 nd_2)}.$$

We now apply the union bound by going over all possible choices of $E_0$ and bounding the probability that it is contained in $\GNd{d_2}$ using Corollary \ref{c:E0contained}. So, Recalling that $d_2$ is large enough as a function of $\varepsilon$, the probability that $G_2$ contains a subset $E_0$ that ruins it is upper bounded by
$$\sum_{m=1}^n\binom{\binom{n}{2}}{m}\left(\frac{Cd_2}{n}\right)^m e^{-\Theta(\varepsilon^2 nd_2)}\leq \sum_{m=1}^n\exp\left(m\ln\frac{eCnd_2}{2m}-\Theta(\varepsilon^2 nd_2)\right)=\exp(-\Theta(\varepsilon^2 nd_2)).$$
\end{proof}
Lemmata \ref{l:boosters2Hamiltonian} and \ref{l:AddingRegG2} imply the following corollary.
\begin{cor}\label{c:AddingRegG2}
For every $\varepsilon>0$ there exists a positive integer $d_0(\varepsilon)$ such that if $d_2\geq d_0$ and $\Gamma$ is a fixed $(n,\beta)$-expander, then $\Prob{r_\ell(\GNd{d_2},\Hamiltonicity_\Gamma)< (1-\varepsilon)d_2/4}\leq e^{-\Theta(\varepsilon^2 nd_2)}$.
\end{cor}

Lemma \ref{l:AddingRegG2} implies our second constraint on the value of $r$, 
\begin{equation}\label{e:rbound2}
r\leq(1-\varepsilon)d_2/4.
\end{equation}
After having proved all of the above, we have all the building blocks needed to complete the proof of Theorem \ref{t:HamResGnd}. Since by Corollary \ref{c:randreg_r_quasirand} the probability that $G_1$ is corrupted is $o(1)$, we can and will condition on the event that $G_1$ is not corrupted. A simple, yet crucial, observation is that if $(G_1-H_1)+(G_2-H_2)$ is not Hamiltonian then $\Gamma+(G_2-H_2)$ is not Hamiltonian for every subgraph $\Gamma\subseteq G_1-H_1$. So, in particular, if there exists such a pair $H_1$ and $H_2$, then by Proposition \ref{p:GminHcontainsGamma} there must exist an $(n,\beta)$-expander $\Gamma$ which spans at most $\mu nd_1$ edges (recalling that $G_1-H_1$ is $(n,d_1,\varepsilon)$-quasi-random by our assumption on $G_1$) for which $\Gamma+ (G_2-H_2)$ is not Hamiltonian. Now we apply Lemma \ref{l:boosters2Hamiltonian} and we can upper bound the probability of the existence of $H_1$ and $H_2$ by the probability there exists a $(n,\beta)$-expander $\Gamma\subseteq G_1$ which spans at most $\mu nd_1$ edges for which there exists a set $E_0\subseteq E(G_2)$ ($E_0$ may depend on $\Gamma$) of cardinality $|E_0|\leq n$ for which every vertex $v\in V$ satisfies $|N_{G_2}(v)\cap B_{\Gamma\cup E_0}(v)|\leq r$. The crux of the proof lies on the fact that now we can apply a union bound argument over a much smaller set of graphs as the graphs we need to go over are much sparser (and hence there are much less of them). As a last note before proceeding to the actual computations we optimize the value of $r$. By \eqref{e:rbound1} and \eqref{e:rbound2} $r$ can be taken to be any constant strictly less than $\min\{d_1/2,d_2/4\}$. Since $d_1+d_2=d$ we choose $d_1=d/3$ and $d_2=2d/3$ which validates the value of $\frac{d}{6}(1-\varepsilon)$ in the statement of the theorem.

Let $\mS$ denote the set of all $(n,\beta)$-expanders on the vertex set $V$ which have between $\mu nd/6$ and $\mu nd/3$ edges. We only need to recall that the edges of $G_1$ and $G_2$ are independent. Then putting everything together implies that $\Prob{r_\ell(\GND,\Hamiltonicity)\leq (1-\varepsilon)d/6}$ is upper bounded by
$$\Prob{G_1\hbox{ is corrupted}} + \cProb{\exists\Gamma\in\mS,\;\Gamma\subseteq G_1 \wedge r_\ell(G_2,\Hamiltonicity_\Gamma)\leq (1-\varepsilon)\frac{d}{6}}{G_1\hbox{ not corrupted}}.$$
Applying the union bound over all possible $\Gamma\in \mS$ and using the fact that the edges of $G_1$ and $G_2$ are independent the above is upper bounded by
\begin{eqnarray}
&& o(1)+\sum_{\Gamma\in\mS}\cProb{\Gamma\subseteq G_1}{G_1\hbox{ not corrupted}}\cdot\Prob{r_\ell(G_2,\Hamiltonicity_\Gamma)\leq (1-\varepsilon)\frac{d}{6}}\nonumber\\
& \leq & o(1)+\sum_{\Gamma\in\mS}\frac{\Prob{\Gamma\subseteq G_1}}{\Prob{G_1\hbox{ not corrupted}}}\cdot\Prob{r_\ell(G_2,\Hamiltonicity_\Gamma)\leq (1-\varepsilon)\frac{d}{6}}.\label{e:res_ham_bound}
\end{eqnarray}
Using Corollary \ref{c:E0contained} we bound the probability that a fixed graph $\Gamma$ spanning $m\leq \mu nd$ edges is contained in $\GNd{d/3}$ by $\left(\frac{Cd}{3n}\right)^m$. The probability that $G_1$ is not corrupted is $1-o(1)$, and lastly, we use Corollary \ref{c:AddingRegG2} to bound the right multiplicand in the sum of \eqref{e:res_ham_bound} as follows:
\begin{eqnarray*}
&& o(1)+(1+o(1))\sum_{m=\mu nd/6}^{\mu nd/3}\binom{\binom{n}{2}}{m}\cdot\left(\frac{Cd}{3n}\right)^m\cdot\exp(-\Theta(\varepsilon^2nd))\\
&\leq& o(1) + (1+o(1))\sum_{m=\mu nd/6}^{\mu nd/3}\left(\frac{Cend}{6m}\right)^m\cdot\exp(-\Theta(\varepsilon^2nd))\\
&\leq& o(1) + \exp\left(\Theta\left(\mu nd\cdot\ln\frac{1}{\mu}\right)-\Theta\left(\varepsilon^2nd\right)\right)=o(1),
\end{eqnarray*}
which completes the proof of the theorem. \hfill$\square$

\section{The Hamiltonicity game played on $\GND$}\label{s:HamGame}
As a closing note for this paper we describe in this short section a new result for the Hamiltonicity game played on the edge-set of a random regular graph of constant degree. This can be viewed as a different type of resilience of the random regular graph with respect to being Hamiltonian. We will need a new definition and some additional structural statements of a typical random regular graph.
\begin{defn}
For every positive $k,\ell$ we say that graph $G=(V,E)$ is a $(k,\ell)$-\emph{magnifier} if every subset of vertices $U\subseteq V$ of cardinality $|U|\leq k$ satisfies $|N_G(U)|\geq \ell\cdot |U|$.
\end{defn}
Much like in Remark \ref{r:nepsexpander} we note that if $G=(V,E)$ is a $(k,\ell)$-magnifier, then every $H=(V,F)$ for $F\supseteq E$ is also a $(k,\ell)$-magnifier. Recalling Definition \ref{d:booster} of boosters, we have the following well-known property of $(k,2)$-magnifiers (see e.g. \cite{FriKri2008}).
\begin{lem}\label{l:magnifier_booster}
Let $G$ be a connected non-Hamiltonian $(k,2)$-magnifier, then $G$ has at least $k^2/2$ boosters.
\end{lem}

Next, we continue with some structural properties of random regular graphs. The following two lemmata will follow immediately from Lemma \ref{l:expmixlem} and Corollaries \ref{c:expmixlem2} and \ref{c:expmixlem3}.
\begin{lem}\label{l:ndl_exp_subgraph}
There exists a positive integer $d_0>0$ such that if $G\sim\GNd{d_1}$ for some $d_1\geq d_0$ and $H$ is a subgraph of $G$ of minimum degree $\delta(H)\geq d_1/5$ then w.h.p. $H$ is an $(n/100,2)$-magnifier.
\end{lem}
\begin{proof}
By Theorem \ref{t:Fri2008} we know that if $d_1$ is large enough then w.h.p. $\lambda(G)\leq\frac{d_1}{60}$, and we hence assume this holds. Let $U\subseteq V$ be a subset of $|U|\leq\frac{n}{100}$ vertices, and assume that $|N_H(U)|<2|U|$. Denote by $W=U\cup N_H(U)$, then our assumption on the minimum degree of $H$ implies that $e_H(W)\geq \frac{d_1|U|}{10}$. On the other hand, Corollary \ref{c:expmixlem3} implies that $e_H(W)\leq e_G(W)\leq\frac{d_1}{n}\binom{3|U|}{2}+\lambda\cdot3|U|\leq d_1|U|\left(\frac{9|U|}{2n}+\frac{1}{20}\right)<\frac{d_1|U|}{10}$, a contradiction which completes the proof of the the lemma.
\end{proof}

For every graph $G$, let $\mB_G$ denote the set of boosters with respect to $G$. Given a graph $H$ on the vertex set of $G$ we say that the set of edges $E_0\subseteq E(G)$  $(H,\alpha)$-\emph{destroys} $G$ if $H\cup E_0$ is non-Hamiltonian and $|E(G)\cap\mB_{H\cup E_0}|<\alpha\cdot|E(G)|$. The following lemma is reminiscent of Lemma \ref{l:AddingRegG2} and can be proved in quite a similar manner.

\begin{lem}\label{l:AddingRegG2_Game}
There exists an integer $d_0>0$ such that if $G_2\sim \GNd{d_2}$ for some fixed $d_2\geq d_0$ and $\Gamma$ is a connected non-Hamiltonian $(\frac{n}{100},2)$-magnifier on the same vertex set, $V$, then the probability there exists a set of at most $n$ edges $E_0\subseteq E(G_2)$ that $(\Gamma,\frac{2}{10^7})$-destroys $G_2$ is at most $e^{-\frac{nd_2}{10^{15}}}$.
\end{lem}
\begin{proof}
We will assume throughout that $d_2$ is large enough (but constant) without computing it explicitly. Fix a choice of at most $n$ pairs of vertices $E_0\subseteq\binom{V}{2}$ such that the graph $\Gamma_2=\Gamma\cup E_0$ is non-Hamiltonian. Our goal is to bound the probability that $G_2\sim\GNd{d_2}$ contains the edges of $E_0$ and that $|E(G_2)\cap\mB_{\Gamma_2}|<\frac{2}{10^7}\cdot|E(G_2)|=\frac{nd_2}{10^7}$. This clearly implies that it suffices to consider only choices of $E_0$ such that the graph $F=(V,E_0)$ satisfies $\Delta(F)\leq d_2$, and hence we proceed with this assumption. Furthermore, if $E_0\subseteq E(G_2)$, then the random graph $\hG=G_2-F$ is distributed according to $\GNd{\bd}$ with degree sequence $\bd=\{d_2-d_F(v)\}_{v\in V}$, conditioned on the event that there are no overlapping edges with $E_0$. Lastly we recall that $\Gamma_2$ is a connected non-Hamiltonian $(\frac{n}{100},2)$-magnifier. Let $X=X(\hG)$ be the random variable equal to $|E(\hG)\cap \mB_{\Gamma_2}|$. As $|E_0|\leq n$, the set $U=\{v\in V\;:\; d_F(v)\geq \frac{d_2}{2}\}$ satisfies $|U|\leq \frac{4|E_0|}{d_2}\leq\frac{4n}{d_2}$, and hence the set $\mB'=\mB_{\Gamma_2}\cap\{\{u,v\}\;:\;u,v\notin U\}$ of boosters with no endpoint in $U$ satisfies $|\mB'|\geq \frac{n^2}{20000}-\frac{4n^2}{d_2}>n^2/10^5$ by Lemma \ref{l:magnifier_booster}. Note that the average degree in $\hG$ clearly satisfies $\od\leq d_2$. Let $u,v\notin U$ to be a pair of distinct vertices, by Proposition \ref{p:edgeprobGNDbar} 
$$\Prob{\{u,v\}\in E(\hG)}\geq(1-o(1))\frac{d_{\hG}(u)\cdot d_{\hG}(v)-d_{\hG}(u)-d_{\hG}(v)}{\od n+ d_{\hG}(u)\cdot d_{\hG}(v)- 2d_{\hG}(u)-2d_{\hG}(v)}\geq \frac{\frac{d_2^2}{5}}{2n d_2}=\frac{d_2}{10n}.$$
We can hence lower bound the expectation of $X$ as follows
$$\Exp{X} \geq |\mB'|\cdot \frac{d_2}{10n}\geq \frac{n^2}{10^5}\cdot\frac{d_2}{10n}= \frac{nd_2}{10^6}.$$
Note that the difference in the value of $X$ for any two graphs with degree sequence $\bd$ on the same vertex set that differ by a single switch can be at most $2$, and we can therefore apply Theorem \ref{t:Azuma4RegGraphs} to derive that
$$\Prob{X(\GNd{\bd})<\frac{nd_2}{10^7}}\leq\Prob{X(\GNd{\bd})<\frac{\Exp{X}}{10}}\leq \exp\left(-\frac{\left(\frac{9}{10}\right)^2\Exp{X}^2}{4nd_2}+\frac{d_2^2-1}{4}+o(1)\right)\leq\exp\left(-\frac{nd_2}{10^{13}}\right).$$
Hence with probability exponentially close to $1$, the random variable $X$ does not deviate much from its expectation. We now estimate $X=X(\hG)$ for $\hG\sim\GNd{\bd}$ conditioned on the event that $\hG$ shares no edge with $F$. Recalling that $\Delta(F)\leq d_2$, Theorem \ref{t:McK85} guarantees that $\Prob{E(\GNd{\bd})\cap E_0=\emptyset} =O(1)$ (as $\gamma=O(1)$ and $\nu=O(1)$), and therefore conditioning on this event can affect the probability that $X(\hG)$ is too small by only a constant factor as follows
$$\cProb{X(\hG)<\frac{nd_2}{10^7}}{E(\hG)\cap E_0=\emptyset}\leq\frac{\Prob{X(\GNd{\bd})<\frac{nd_2}{10^7}}}{\Prob{E(\GNd{\bd})\cap E_0=\emptyset}}\leq \exp\left(-\frac{nd_2}{10^{14}}\right).$$
To complete the proof we apply the union bound by going over all possible choices for the set $E_0$ and bounding the probability that it is contained in $G_2$ using Corollary \ref{c:E0contained}. Note that although some restrictions are set on $E_0$ (i.e. creating a non-Hamiltonian $\Gamma_2$ and satisfying $\Delta(F)\leq d_2$), these are not taken into account in the union bound where we simply go over all possible subsets of at most $n$ pairs of vertices from $V$. So, given a fixed $\Gamma$ as above the probability that $G_2$ contains a subset of at most $n$ edges that $(\Gamma,\frac{2}{10^7})$-destroys it is upper bounded by
\begin{eqnarray*}
\sum_{m=1}^{n}\binom{\binom{n}{2}}{m}\left(\frac{Cd_2}{n}\right)^m\cdot\exp\left(-\frac{nd_2}{10^{14}}\right)\leq\sum_{m=1}^{n}\exp\left(m\ln\frac{eCnd_2}{2m}-\frac{nd_2}{10^{14}}\right)\leq\exp\left(-\frac{nd_2}{10^{15}}\right).
\end{eqnarray*}
\end{proof}

\begin{cor}\label{c:AddingRegG2_Game}
There exists an integer $d_0>0$ such that if $G_1\sim \GNd{d_1}$ and $G_2\sim \GNd{d_2}$ (sampled on the same vertex set) for some fixed $d_1$ and $d_2$ and where $d_0\leq d_1\ll d_2$, then w.h.p. $G_1$ does not contain a connected non-Hamiltonian $(\frac{n}{100},2)$-magnifier subgraph $\Gamma$ satisfying $e(\Gamma)\leq \frac{n(d_1+1)}{5}$ for which there exists a set of at most $n$ edges $E_0\subseteq E(G_2)$ that $(\Gamma,\frac{2}{10^7})$-destroys $G_2$.
\end{cor}
\begin{rem}
In the above lemma by $d_1\ll d_2$ we mean that we can choose those two values such that the (constant) ratio $d_2/d_1$ is chosen to be large enough for the argument to go through.
\end{rem}
\begin{proof}
We note that there are at most $\binom{\binom{n}{2}}{m}$ graphs on this vertex set that span $m$ edges, and if $m\leq \frac{n(d_1+1)}{5}$ by Corollary \ref{c:E0contained} the probability that each of these graphs is contained in $G_1$ is at most $\left(\frac{Cd_1}{n}\right)^m$. To get the above mentioned result we apply the union bound over all the possible connected non-Hamiltonian $(\frac{n}{100},2)$-magnifier graphs on the specified vertex set with the specified number of edges (in our computation below we will actually just go over all possible choices of $m\leq \frac{n(d_1+1)}{5}$ pairs of vertices from $V$). For every such graph, $\Gamma$, we upper bound the probability that it is both contained in $G_1$ and that there exists some $E_0$ of at most $n$ edges that $(\Gamma,\frac{2}{10^7})$-destroys $G_2$. The latter probability is upper bounded using Lemma \ref{l:AddingRegG2_Game}, and we note that the above two events are independent (due to the independence of the edges of $G_1$ and $G_2$). So, recalling that $d_2\gg d_1$ the probability that the conditions of the lemma are not satisfied is upper bounded by
\begin{eqnarray*}
\sum_{m=1}^{\frac{n(d_1+1)}{5}}\binom{\binom{n}{2}}{m}\cdot\left(\frac{Cd_1}{n}\right)^m\cdot\exp\left(-\frac{nd_2}{10^{15}}\right)\leq \sum_{m=1}^{\frac{n(d_1+1)}{5}}\exp\left(m\ln\frac{eCnd_1}{2m}-\frac{nd_2}{10^{15}}\right)=o(1)
\end{eqnarray*}
as claimed.
\end{proof}
Before we proceed to the proof of Theorem \ref{t:HamGame} we quote the following results in the context of Maker-Breaker games. Denote by $\delta_k$ the graph property of having minimum degree at least $k$.
\begin{lem}[Hefetz et. al. \cite{HefEtAl2009}, Lemma 10]\label{l:quarter_degree}
For any positive integer $k$ and graph $G$ on $n$ vertices, if $\delta(G)\geq 4k$ then $G\in\mathcal{M}_{\delta_k}$. Moreover, Maker can win this game in at most $kn$ moves.
\end{lem}
\begin{thm}[Lehman \cite{Leh64}]\label{t:Leh64}
For every graph $G$ it holds that $G\in\mathcal{M}_{\VertConn{1}}$ if and only if $G$ admits two edge-disjoint spanning trees.
\end{thm}

Having acquired all the necessary building blocks we move to describe the proof of Theorem \ref{t:HamGame}.
\begin{proof}[Proof of Theorem \ref{t:HamGame}]
We assume $d$ is large enough and set $\varepsilon=10^{-7}$. Let $d_0$ denote the constant from Corollary \ref{c:AddingRegG2_Game} and let $d_1-4\geq d_0$ be such that $d_2=d-d_1>\frac{d_1+6}{5\varepsilon}$. Let $G\sim\GNd{d_1}\oplus\GNd{d_2}\approx\GNd{4}\oplus\GNd{d_1-4}\oplus\GNd{d_2}$ and recall that Theorem \ref{t:Jan95} implies that it suffices to prove the statement in this probability space. We can assume $G$ can be decomposed into two graphs $G=G_1+G_2$ with disjoint edge sets such that $G_1$ and $G_2$ are $d_1$-regular and $d_2$-regular graphs respectively that satisfy the property described by Corollary \ref{c:AddingRegG2_Game}. Moreover, we can also assume using Theorems \ref{t:Jan95} and \ref{t:KimWor2001} that in turn $G_1$ can be decomposed into two graphs $G_1=G_{1,1}+G_{1,2}$ where  $G_{1,1}$ is $4$-regular and is composed of two disjoint Hamilton cycles and $G_{1,2}$ is $(d_1-4)$-regular and satisfies the property described by Lemma \ref{l:ndl_exp_subgraph}.

Maker's strategy is thus quite natural. Let $e_i$ denote the edge selected by Maker in the $i^{th}$ turn and by $M_i=(V,\{e_1,\ldots,e_i\})$ the graph Maker possesses after $i$ turns. For the first $t_1\leq n(d_1-4)/5+n=n(d_1+1)/5$ turns of the game, Maker plays solely on the edges of $G_1$. During this phase, Maker plays two games in parallel, one on the edge set of $G_{1,1}$ and the other on the edge set of $G_{1,2}$, i.e. if Breaker takes an edge from $G_{1,x}$ in turn $i\leq t_1$, then Maker responds by taking an edge from the same graph, and if Breaker takes an edge from $G_2$ then Maker responds by taking an edge from $G_1$ which advances him to his goal in either of the games. The goal of Maker playing on $G_{1,1}$ is to create a connected graph. As $G_{1,1}$ is composed of two edge-disjoint Hamilton cycles, by Theorem \ref{t:Leh64} Maker can win this game using at most $n$ moves. The goal of Maker playing on $G_{1,2}$ is to build a graph $H$ of minimum degree $\delta(H)\geq (d_1-4)/5$, and using Lemma \ref{l:quarter_degree} Maker can win this game using at most $n(d_1-4)/5$ moves. By Lemma \ref{l:ndl_exp_subgraph} Maker obtains this way a $(\frac{n}{100},2)$-magnifier. It follows by the properties of $G_1$ that the graph $M_{t_1}$ (which contains the union of graph in Maker's possession after having won the two games) is both connected and a $(\frac{n}{100},2)$-magnifier.

After having completed the construction of $M_{t_1}$ Maker moves to the second phase of his strategy. For the next $t_2\leq n$ turns Maker will select edges from $G_2$ which are boosters with respect to the graph he possesses at each turn. Let $t_1< i\leq t_1+t_2$, and let $E_i=\{e_{t_1+1},\ldots, e_{i-1}\}\subseteq E(G_2)$ denote the set of edges chosen by Maker in the $i-1-t_1$ rounds of the second phase. By this turn Breaker (who moves first) has taken $i\leq t_1+n\leq n(d_1+6)/5<\varepsilon nd_2$ edges from $G_2$. The assumption that $G_2$ satisfies the condition of Corollary \ref{c:AddingRegG2_Game} implies that the set $E_i$ does not $(M_{t_1},2\varepsilon)$-destroy $G_2$, and hence the graph $M_{i}$ has at least $2\varepsilon\cdot e(G_2)=\varepsilon nd_2$ boosters among the edges of $G_2$. As Breaker could not have taken them all by this turn, Maker can freely choose a booster in the $i^{th}$ turn. This implies that either $M_i$ is Hamiltonian or that $\ell(M_i)>\ell(M_{i-1})$. Maker can continue playing this second phase for at least $n$ turns, and hence will finish by creating a Hamilton cycle before Breaker can stop him.
\end{proof}

\section{Discussion and concluding remarks}\label{s:Conclusion}
In this paper we considered the local resilience of the typical random regular graph of fixed degree with respect to several graph properties, and studied the Hamiltonicity game played the edges of a random $d$-regular graph. The ideological similarities of the local resilience of the Hamiltonicity property and the Maker-Breaker game played for this property enabled us to tackle both using similar techniques. The following are some additional related issues, extensions, and open problems we believe would be interesting to further study in this context.

\begin{itemize}
\item
A natural way to extend Theorem \ref{t:HamResGnd} would be to not restrict $d$ to be constant but to let it grow with $n$. In \cite{SudVu2008} Sudakov and Vu proved that if $G$ is an $(n,d,\lambda)$-graph with $d/\lambda\geq \log^{1+\delta}n$ for any $\delta>0$ then $r_\ell(G,\Hamiltonicity)\leq (1/2-\varepsilon)n$ for any $\varepsilon>0$. Although this spectral gap cannot be attained for $d=O(\log^{2(1+o(1))}n)$ by the Alon-Boppana Theorem (\cite{Nil91}), when $d=\Omega(\log^{2+\delta}n)$ for any $\delta>0$ this (quite moderate) condition on the second eigenvalue in a typical graph of $\GND$\footnote{See \cite{BroEtAl99} and \cite{KriEtAl2001} to cover most of this range of $d$. The range not covered by the previous citations can be dealt with using standard techniques for bounding eigenvalues of random graphs (see e.g. \cite{KriSud2003}).} is satisfied. We believe that techniques similar to those applied in the current paper can be used to show that in the missing range $1\ll d\ll\log^{2(1+o(1))}n$ the local resilience of $\GND$ is w.h.p. at least $(1-\varepsilon)d/6$.
\item It would be interesting to further investigate the local resilience of the typical graph in both $\GND$ for constant values of $d$ and $\GNP$ with $p=\frac{K\ln n}{n}$. We believe, as was previously conjectured by Sudakov and Vu for the $\GNP$ case in \cite{SudVu2008}, that the true order of this parameter is closer to the upper bound than the lower bound.
\begin{conj}
For every $\varepsilon>0$ there exists an integer $d_0(\varepsilon)>0$ such that for every fixed integer $d \geq d_0$ w.h.p.
$$\left|r_\ell(\GND,\Hamiltonicity)-\frac{d}{2}\right|\leq\varepsilon d.$$
\end{conj}
\begin{conj}
For every $\varepsilon>0$ there exists an integer $K(\varepsilon)>0$ such that for every $p \geq \frac{K\ln n}{n}$ w.h.p.
$$\left|r_\ell(\GNP,\Hamiltonicity)-\frac{np}{2}\right|\leq\varepsilon np.$$
\end{conj}

\item Theorem \ref{t:HamGame} does not find the minimal value of $d$ for which it is true that w.h.p. $\GND\in \mathcal{M}_{\Hamiltonicity}$. It follows from a result of Hefetz et. al \cite{HefEtAlPre} that $d$ is at least $5$ (while recalling that $\GND\in \Hamiltonicity$ for all fixed $d\geq 3$), but computing this minimal value exactly seems to require new ideas.

\item The Positional Game result in this work deals specifically with the Hamiltonicity game, but there are quite some other natural properties for which one could ask the same question; for example, edge and vertex connectivity. One of the first and fundamental results about the $\GND$ model is that of Bollob\'{a}s \cite{Bol81} and Wormald \cite{Wor81} states that for fixed $d\geq 3$ w.h.p. $\GND\in \VertConn{d}$. Theorem \ref{t:HamGame} clearly implies that for large enough values of fixed $d$ w.h.p. $\GND\in \mathcal{M}_{\VertConn{2}}$,
although this is far from being optimal. An interesting question is the following:
\begin{prob}
Determine the maximal $k=k(d)$ for which w.h.p. $\GND\in \mathcal{M}_{\VertConn{k}}$.
\end{prob}
Using techniques somewhat similar to the proof of Theorem \ref{t:HamGame} (i.e. splitting the base graph into two random graphs and playing on the first one to get some expansion properties which guarantee small sets have large enough neighborhoods, and playing on the second one to guarantee linear size sets have large enough neighborhoods) should imply that $k\geq cd$ for some universal constant $c>0$, but finding the optimal value of $k$ may require some further research.
\end{itemize}

\subsection*{Acknowledgment}
We would like to thank the referees for their valuable comments and remarks.


\begin{thebibliography}{10}

\bibitem{AloSpe2008}
N.~Alon and J.~H. Spencer.
\newblock {\em The Probabilistic Method}.
\newblock Wiley-Interscience Series in Discrete Mathematics and Optimization.
  John Wiley \& Sons, third edition, 2008.

\bibitem{Beck2008}
J.~Beck.
\newblock {\em Combinatorial Games: {T}ic-{T}ac-{T}oe theory}.
\newblock Cambridge University Press, New York, 2008.

\bibitem{Bol81}
B.~Bollob\'{a}s.
\newblock Random graphs.
\newblock In H.~N.~V. Temperley, editor, {\em Combinatorics}, volume~52 of {\em
  London Mathematical Society Lecture Note Series}, pages 80--102. Cambridge
  University Press, 1981.

\bibitem{Bol83}
B.~Bollob\'{a}s.
\newblock Almost all regular graphs are {H}amiltonian.
\newblock {\em European Journal of Combinatorics}, 4:94--106, 1983.

\bibitem{Bol2001}
B.~Bollob\'{a}s.
\newblock {\em Random Graphs}.
\newblock Cambridge University Press, 2001.

\bibitem{BroEtAl99}
A.~Broder, A.~Frieze, S.~Suen, and E.~Upfal.
\newblock Optimal construction of edge-disjoint paths in random graphs.
\newblock {\em {SIAM} Journal on Computing}, 28(2):541--573, 1999.

\bibitem{Chung2004}
F.~Chung.
\newblock Discrete isoperimetric inequalities.
\newblock In A.~Grigor'yan and S.~T. Yau, editors, {\em Surveys in Differential
  Geometry: Eigenvalues of Laplacians and other geometric operators},
  volume~IX. International Press, 2004.

\bibitem{ChvErd78}
V.~Chv\'{a}tal and P.~Erd\H{o}s.
\newblock Biased positional games.
\newblock {\em Annals of Discrete Mathematics}, 2:221--228, 1978.

\bibitem{DelEtAl2008}
D.~Dellamonica, Y.~Kohayakawa, M.~Marciniszyn, and A.~Steger.
\newblock On the resilience of long cycles in random graphs.
\newblock {\em The Electronic Journal of Combinatorics}, 15:R32, 2008.

\bibitem{Die2005}
R.~Diestel.
\newblock {\em Graph theory}, volume 173 of {\em Graduate Texts in
  Mathematics}.
\newblock Springer-Verlag, third edition, 2005.

\bibitem{FenFri84}
T.~Fenner and A.~Frieze.
\newblock {H}amiltonian cycles in random regular graphs.
\newblock {\em Journal of Combinatorial Theory, Series B}, 37(2):103--198,
  1984.

\bibitem{Fri2008}
J.~Friedman.
\newblock A proof of {A}lon's second eigenvalue conjecture and related
  problems.
\newblock {\em Memoirs of the {AMS}}, 195(910), 2008.

\bibitem{Fri88}
A.~Frieze.
\newblock Finding {H}amilton cycles in sparse random graphs.
\newblock {\em Journal of Combinatorial Theory, Series B}, 44(2):230--250,
  1988.

\bibitem{FriKri2008}
A.~Frieze and M.~Krivelevich.
\newblock On two {H}amiltonian cycle problems in random graphs.
\newblock {\em Israel Journal of Mathematics}, 166:221--234, 2008.

\bibitem{GreEtAl2002}
C.~Greenhill, J.~H. Kim, S.~Janson, and N.~C. Wormald.
\newblock Permutation pseudographs and contiguity.
\newblock {\em Combinatorics, Probability, and Computing}, 11(3):273--298,
  2002.

\bibitem{HefEtAl2009}
D.~Hefetz, M.~Krivelevich, M.~Stojakovi\'{c}, and T.~Szab\'{o}.
\newblock A sharp threshold for the {H}amilton cycle {M}aker-{B}reaker game.
\newblock {\em Random Structures and Algorithms}, 34(1):112--122, 2009.

\bibitem{HefEtAlPre}
D.~Hefetz, M.~Krivelevich, M.~Stojakovi\'{c}, and T.~Szab\'{o}.
\newblock Global {M}aker-{B}reaker games on sparse graphs.
\newblock preprint.

\bibitem{HefKriSza2009}
D.~Hefetz, M.~Krivelevich, and T.~Szab\'{o}.
\newblock {H}amilton cycles in highly connected and expanding graphs.
\newblock {\em Combinatorica}, 29(5):547--568, 2009.

\bibitem{HefSti2009}
D.~Hefetz and S.~Stich.
\newblock On two problems regarding the {H}amilton cycle game.
\newblock {\em The Electronic Journal of Combinatorics}, 16(1):R28, 2009.

\bibitem{Jan95}
S.~Janson.
\newblock Random regular graphs: asymptotic distributions and contiguity.
\newblock {\em Combinatorics, Probability, and Computing}, 4(4):369--405, 1995.

\bibitem{JanLucRuc2000}
S.~Janson, T.~{\L}uczak, and A.~Ruci{\'n}ski.
\newblock {\em Random Graphs}.
\newblock Wiley-Interscience Series in Discrete Mathematics and Optimization.
  John Wiley \& Sons, 2000.

\bibitem{KimSudVu2002}
J.~H. Kim, B.~Sudakov, and V.~H. Vu.
\newblock On the asymmetry of random regular graphs.
\newblock {\em Random Structures and Algorithms}, 21(3--4):216--224, 2002.

\bibitem{KimSudVu2007}
J.~H. Kim, B.~Sudakov, and V.~H. Vu.
\newblock Small subgraphs of random regular graphs.
\newblock {\em Discrete Mathematics}, 307(15):1961--1967, 2007.

\bibitem{KimVu2004}
J.~H. Kim and V.~H. Vu.
\newblock Sandwiching random graphs.
\newblock {\em Advances in Mathematics}, 188(2):444--469, 2004.

\bibitem{KimWor2001}
J.~H. Kim and N.~C. Wormald.
\newblock Random matchings which induce {H}amilton cycles, and {H}amiltonian
  decompositions of random regular graphs.
\newblock {\em Journal of Combinatorial Theory Series B}, 81(1):20--44, 2001.

\bibitem{KriLeeSud2010}
M.~Krivelevich, C.~Lee, and B.~Sudakov.
\newblock Resilient pancyclicity of random and pseudo-random graphs.
\newblock {\em {SIAM} Journal on Discrete Mathematics}, 24(1):1--16, 2010.

\bibitem{KriSud2003}
M.~Krivelevich and B.~Sudakov.
\newblock Sparse pseudo-random graphs are {H}amiltonian.
\newblock {\em Journal of Graph Theory}, 42(1):17--33, 2003.

\bibitem{KriSud2006}
M.~Krivelevich and B.~Sudakov.
\newblock Pseudo-random graphs.
\newblock In E.~Gy\H{o}ri, G.~O.~H. Katona, and L.~Lov\'{a}sz, editors, {\em
  More sets, graphs and numbers: A Salute to Vera S\`{o}s and Andr\'{a}s
  Hajnal}, volume~15 of {\em Bolyai Society Mathematical Studies}, pages
  199--262. Springer, 2006.

\bibitem{KriEtAl2001}
M.~Krivelevich, B.~Sudakov, V.~H. Vu, and N.~Wormald.
\newblock Random regular graphs of high degree.
\newblock {\em Random Structures and Algorithms}, 18(4):346--363, 2001.

\bibitem{Leh64}
A.~Lehman.
\newblock A solution of the {S}hannon switching game.
\newblock {\em Journal of the Society for Industrial and Applied Mathematics},
  12(4):687--725, 1964.

\bibitem{McD98}
C.~McDiarmid.
\newblock Concentration.
\newblock In M.~Habib, C.~McDiarmid, J.~Ramirez-Alfonsin, and B.~Reed, editors,
  {\em Probabilistic Methods for Algorithmic Discrete Mathematics}, pages
  195--248. Springer, 1998.

\bibitem{McK85}
B.~D. McKay.
\newblock Asymptotics for symmetric 0-1 matrices with prescribed row sums.
\newblock {\em Ars Combinatorica}, 19(A):15--26, 1985.

\bibitem{Nil91}
A.~Nilli.
\newblock On the second eigenvalue of a graph.
\newblock {\em Discrete Mathematics}, 91(2):207--210, 1991.

\bibitem{Pos76}
L.~P\'{o}sa.
\newblock Hamiltonian circuits in random graphs.
\newblock {\em Discrete Mathematics}, 14(4):359--364, 1976.

\bibitem{RobWor92}
R.~Robinson and N.~Wormald.
\newblock Allmost all cubic graphs are {H}amiltonian.
\newblock {\em Random Structures and Algorithms}, 3(2):117--125, 1992.

\bibitem{RobWor94}
R.~Robinson and N.~Wormald.
\newblock Allmost all regular graphs are {H}amiltonian.
\newblock {\em Random Structures and Algorithms}, 5(2):363--374, 1994.

\bibitem{SudVu2008}
B.~Sudakov and V.~H. Vu.
\newblock The local resilience of random graphs.
\newblock {\em Random Structures and Algorithms}, 33(4):409--433, 2008.

\bibitem{Tur41}
P.~Tur{\'a}n.
\newblock Eine {E}xtremalaufgabe aus der {G}raphentheorie.
\newblock {\em Mat. Fiz. Lapok}, 48:436--452, 1941.

\bibitem{Wor81}
N.~C. Wormald.
\newblock The asymptotic connectivity of labelled regular graphs.
\newblock {\em Journal of Combinatorial Theory, Series B}, 31:156--167, 1981.

\bibitem{Wor99}
N.~C. Wormald.
\newblock Models of random regular graphs.
\newblock In J.~Lamb and D.~Preece, editors, {\em Surveys in Combinatorics},
  volume 276 of {\em London Mathematical Society Lecture Note Series}, pages
  239--298. Cambridge University Press, 1999.

\end{thebibliography}

\appendix
\section{Local resilience of $\GNP$ - proof of Theorem \ref{t:HamResGnp}}
This subsection is devoted to demonstrate how to establish a lower bound on the local resilience with respect to Hamiltonicity in the case of $\GNP$. The proof of this case follows closely the steps of the proof of Theorem \ref{t:HamResGnd}. We will be using the same notation as in the proof of Theorem \ref{t:HamResGnd} and our focus will be on the adaptations needed for the $\GNP$ case. So, we may be somewhat brief in the explanation of the arguments to avoid repetition. We start by proving some expansion properties that a typical graph in $\GNP$ satisfies.
\begin{lem}\label{l:gnp_r_quasirand}
For every $0< \varepsilon \leq 1$ there exists a constant $K=K(\varepsilon)$ such that if $G\sim \GNp{p_1}$ for some $p_1\geq K\ln n/n$, then w.h.p. for any subgraph $H\subseteq G$ of maximum degree $\Delta(H)\leq (1-\varepsilon)np_1/2$, the graph $G-H$ is $(n,np_1,\varepsilon)$-quasi-random.
\end{lem}
\begin{proof}
The lemma will follow from the following series of claims.
\begin{clm}
W.h.p. $\delta(G-H)\geq (1-\varepsilon)np_1/2$.
\end{clm}
\begin{proof}
For every $v\in V$, $d_G(v)\sim \Bin(n-1,p_1)$. Clearly, by Chernoff (Theorem \ref{t:Chernoff} item \ref{i:Chernoff2})
$$\Prob{d_{G-H}(v)<(1-\varepsilon)np_1/2}\leq\Prob{d_{G}(v)<(1-\varepsilon)np_1}=o(1/n),$$
and using the union bound over all vertices completes the proof.
\end{proof}

\begin{clm}
W.h.p. $\Delta(G)\leq (1+\varepsilon)np_1$.
\end{clm}
\begin{proof}
Similarly, for every $v\in V$, $d_G(v)\sim \Bin(n-1,p_1)$. Again, by Chernoff
$$\Prob{d_{G}(v)<(1+\varepsilon)np_1}=o(1/n),$$
and using the union bound over all vertices completes the proof.
\end{proof}

\begin{clm}
W.h.p. every $U\subseteq V$ of cardinality $|U|< \mu n/14$ satisfies $e_{G-H}(U)\leq\mu |U|np_1/14$.
\end{clm}
\begin{proof}
Fixing such a subset of vertices $U$, we have that $e_G(U)\sim \Bin(\binom{|U|}{2},p_1)$ and hence by Theorem \ref{t:Chernoff} item \ref{i:Chernoff1} we get that
$$\Prob{e_{G-H}(U)>\mu|U|np_1/14}\leq\Prob{e_G(U)>\mu|U|np_1/14}\leq e^{-\Theta(\mu|U|np_1)}.$$
To upper bound the probability of the existence of a subset of vertices for which the assertion of the claim does not hold, we use the union bound (and recall that $K$ is large enough)
\begin{eqnarray*}
\sum_{t=1}^{\mu n/14}\binom{n}{t}e^{-\Theta(\mu tnp_1)}&\leq&\sum_{t=1}^{\mu n/14}\exp(t\left(\ln\frac{en}{t}-\Theta(\mu np_1))\right)\\
&\leq&\mu n/14\cdot\exp\left(\mu n/14\left(\ln \frac{14e}{\mu}-\Theta(\mu np_1)\right)\right)=o(1).
\end{eqnarray*}
\end{proof}
\begin{clm}
W.h.p. every two disjoint subsets $U,W\subseteq V$ where $\beta n\leq |U|< 2\beta n $ and $|W|\geq n/3$ satisfy $e_{G-H}(U,W)\geq(1-\varepsilon/4)p_1|U||W|-(1-\varepsilon)\frac{np_1}{2}|U|$.
\end{clm}
\begin{proof}
Fixing two such subsets of vertices $U$ and $W$, we have that $e_G(U,W)\sim\Bin(|U||W|,p_1)$ and hence by Chernoff (Theorem \ref{t:Chernoff} item \ref{i:Chernoff2}) we get that $\Prob{e_G(U,W)<(1-\varepsilon/4)|U||W|p_1}\leq e^{-\Theta(\varepsilon^2|U||W|p_1)}=e^{-\Theta(\varepsilon^2\mu n^2p_1)}=o(4^{-n})$, and as $e_{G-H}(U,W) \geq e_G(U,W)-\Delta(H)\cdot|U|$, applying the union bound completes the proof.
\end{proof}
Recalling Definition \ref{d:quasirand} completes the proof of the lemma.
\end{proof}
If $G_1\sim \GNp{p_1}$ does not satisfy the statement of the Lemma \ref{l:gnp_r_quasirand} we say that $G_1$ is corrupted. Recalling the definitions of $\mu(\varepsilon)$ and $\beta(\varepsilon)$ in \eqref{e:defnofmuandbeta}, Proposition \ref{p:GminHcontainsGamma} implies that if $G_1$ is not corrupted then for every $H_1\subseteq G_1$ of maximum degree $\Delta(H_1)\leq (1-\varepsilon)np_1/2$ the graph $G_1-H_1$ contains a $(n,\beta)$-expander subgraph that spans at most $\mu n^2p_1$ edges. We continue to the next phase of our proof, namely, showing that for large enough values of $p_2$ adding on top of a fixed $(n,\beta)$-expander the edges of $\GNp{p_2}$ the resulting graph will be Hamiltonian with probability exponentially close to $1$.

\begin{lem}\label{l:AddingG2}
For every $0<\varepsilon\leq 1$ there exists a large enough constant $K=K(\varepsilon)$ such that if $p_2\geq  K\ln n/n$ and $\Gamma$ is a fixed  $(n,\beta)$-expander on the same vertex set, then $\Prob{r_\ell(\GNp{p_2},\Hamiltonicity_\Gamma)< (1-\varepsilon)np_2/4}\leq e^{-\Theta(\varepsilon^2n^2p_2)}$.
\end{lem}
\begin{proof}
Throughout we assume that $\varepsilon$ is small enough and $K$ is large enough (as a function of $\varepsilon$), without giving explicit bounds on them. Fix a choice of at most $n$ pairs of vertices $E_0\subseteq \binom{V}{2}$, and let $G_2\sim\GNp{p_2}$. By Setting $d_2=np_2$, we say that $E_0$ ruins $G_2$ as defined in Lemma \ref{l:AddingRegG2}. Let $\Gamma_2=\Gamma\cup E_0$, then by Lemma \ref{l:nEpsExpander} the set $A_0=\{v\in V: |B_{\Gamma_2}(v)|\geq n/4\}$ must satisfy $|A_0|\geq n/4$, and hence the random variable $X=\sum_{v\in A}|N_{G_2}(v)\cap B_{\Gamma_2}(v)|$ satisfies $\Exp{X}\geq\frac{n^2p_2}{16}$. Note that for any two graphs on the vertex set $V$ that differ by a single edge, their value of $X$ can change by at most $2$ ($1$ for every endpoint of the edge), hence we can apply the Azuma-Hoeffding inequality for martingales of bounded variance (see e.g. \cite[Theorem 7.4.3]{AloSpe2008}), to prove that $X$ is concentrated around its expectation. In the process of ``exposing'' the edges of the graph, it suffices to expose only the edges from vertices in $v\in A_0$ to their respective sets $B_{\Gamma_2}(v)$. This implies that the total variance of the martingale is upper bounded by $\frac{n^2}{16}p_2(1-p_2)$, and hence
\begin{equation*}
\Prob{X(G_2)\leq (1-\varepsilon)\frac{n^2p_2}{16}}\leq \exp\left(-\Theta(\varepsilon^2 n^2p_2)\right).
\end{equation*}
We stress the fact that after fixing $E_0$ the value of $X(G_2)$ and the event that $E_0\subseteq E(G_2)$ are independent. We conclude by applying the union bound over all possible choices of $E_0$,
\begin{eqnarray*}
\Prob{\exists E_0\hbox{ ruins } G_2}
&\leq& \sum_{E_0}\Prob{E_0\subseteq E(G_2)}\cdot\Prob{X(G_2)\leq (1-\varepsilon)\frac{n^2p_2}{16}}\\
&\leq& \sum_{m=1}^n\binom{\binom{n}{2}}{m}\cdot p_2^m\cdot e^{-\Theta(\varepsilon^2 n^2p_2)}\leq \sum_{m=1}^n\exp\left(m\ln\frac{en^2p_2}{2m}-\Theta(\varepsilon^2 n^2p_2)\right)\\
&=&\exp(-\Theta(\varepsilon^2n^2p_2)),
\end{eqnarray*}
and Lemma \ref{l:boosters2Hamiltonian} completes the proof.
\end{proof}

We conclude this section by providing the proof of Theorem \ref{t:HamResGnp} which we restate here for clarity.
\begin{thm}
For every $\varepsilon>0$ there exists a large enough constant $K=K(\varepsilon)$ such that if $G=(V,E)$ is a graph sampled from $\GNP$ for some $p\geq K\ln n/n$ then w.h.p. for every subgraph $H\subseteq G$ satisfying $\Delta(H)\leq(1-\varepsilon)np/6$, the graph $G-H$ is Hamiltonian.
\end{thm}
\begin{proof}
It is immediate to see that the graph $G$ can be generated from the $\GNP$ distribution as follows: Let $1-p=(1-p_1)(1-p_2)$, and for every pair of vertices $\{u,v\}$ select it to be in the graph $G_1$ with probability $p_1$ and independently in the graph $p_2$ with probability $p_2$. Clearly, $G_1\sim \GNp{p_1}$ and $G_2\sim \GNp{p_2}$ and taking $G=G_1+G_2$ (where the edge set is taken as a union of sets, i.e., parallel edges are taken as one edge) we get a graph distributed according to $\GNP$. It thus suffices to prove the claim under these settings. Much like in the proof of Theorem \ref{t:HamResGnd} we optimize the values of $p_2=2p_1$ and we get $p=3p_1(1+o(1))$. Let $\mS$ denote the set of all $(n,\beta)$-expanders on the vertex set $V$ which have at most $\mu n^2p_1$ edges. Lastly, we stress that the edges of $G_1$ and $G_2$ are independent. Combining all of the above we have that $\Prob{r_\ell(\GNP,\Hamiltonicity)\leq (1-\varepsilon)np/6}$ is upper bounded by
\begin{eqnarray*}
&&\Prob{G_1\hbox{ is corrupted}} + \sum_{\Gamma\in\mS} \cProb{\Gamma\subseteq G_1 \wedge r_\ell(G_2,\Hamiltonicity_\Gamma)\leq (1-\varepsilon)\frac{np_2}{4}}{G_1\hbox{ not corrupted}}\\
& = & o(1)+\sum_{\Gamma\in\mS}\cProb{\Gamma\subseteq G_1}{G_1\hbox{ not corrupted}}\cdot\Prob{r_\ell(G_2,\Hamiltonicity_\Gamma)\leq (1-\varepsilon)\frac{np_2}{4}}\\
& \leq & o(1)+\sum_{\Gamma\in\mS}\frac{\Prob{\Gamma\subseteq G_1}}{\Prob{G_1\hbox{ not corrupted}}}\cdot\Prob{r_\ell(G_2,\Hamiltonicity_\Gamma)\leq (1-\varepsilon)\frac{np_2}{4}}\\
&\leq & o(1) +(1+o(1))\sum_{m=1}^{\mu n^2p_1}\binom{\binom{n}{2}}{m}\cdot p_1^m\cdot\exp(-\Theta(\varepsilon^2n^2p_2))\\
&\leq & o(1) + (1+o(1))\sum_{m=1}^{\mu n^2p_1}\left(\frac{en^2p_1}{2m}\right)^m\cdot\exp(-\Theta(\varepsilon^2n^2p_2))\\
&\leq & o(1) + \exp(\Theta(\mu n^2p_1\ln\frac{1}{\mu})-\Theta(\varepsilon^2n^2p_2)) = o(1).
\end{eqnarray*}
\end{proof}
\end{document}